\documentclass[3p,11pt]{elsarticle}
\usepackage{lineno}
\modulolinenumbers[1]
\usepackage{geometry}
\usepackage{graphicx}
\usepackage[usenames,dvipsnames]{color}
\usepackage[english]{babel}
\usepackage{epstopdf}
\usepackage{latexsym}
\usepackage{url}
\usepackage{amstext}
\usepackage{amssymb}
\usepackage{amsmath,amsthm}
\usepackage{pifont}
\usepackage{ulem}
\usepackage[pagebackref=false]{hyperref}
\hypersetup{
    colorlinks=true,
    linkcolor=blue,
    filecolor=magenta,
    urlcolor=blue,
}
\graphicspath{ {Images/} }
\usepackage{times}
\usepackage{wrapfig}
\usepackage{float}
\usepackage{tikz}
\usepackage{pgfplots}
\usepackage{caption}
\usepackage{subcaption}
\newtheorem{theorem}{Theorem}

\begin{document}

\begin{frontmatter}

\title{Pattern alternations induced by nonlocal interactions}

\author[inst1]{Swadesh Pal}
\ead{spal@wlu.ca}
\author[inst1,inst2]{Roderick Melnik\corref{cor1}}
\ead{rmelnik@wlu.ca}
\author[inst3]{Malay Banerjee}
\ead{malayb@iitk.ac.in}

\cortext[cor1]{Corresponding author}
\address[inst1]{MS2Discovery Interdisciplinary Research Institute, Wilfrid Laurier University, Waterloo, Canada}
\address[inst2]{BCAM - Basque Center for Applied Mathematics, E-48009, Bilbao, Spain}
\address[inst3]{Department of Mathematics and Statistics, IIT Kanpur, Kanpur, India}

\begin{abstract}
Pattern formation is a visual understanding of the dynamics of complex systems. Patterns arise in many ways, such as the segmentation of animals, bacterial colonies during growth, vegetation, chemical reactions, etc. In most cases, the long-range diffusion occurs, and the usual reaction-diffusion (RD) model can not capture such phenomena. The nonlocal RD model, on the other hand, can fill the gap. Analytical derivation of the amplitude equations (AE) for an RD system is a valuable tool to predict the pattern selections, in particular, the stationary Turing patterns when they occur. In this paper, we analyze the conditions for the Turing bifurcation for the nonlocal model and also derive the AE for the nonlocal RD model near the Turing bifurcation threshold to describe the reason behind the pattern selections. This derivation of the AE is not only limited to the nonlocal prey-predator model, as shown in our representative example but also can be applied to other nonlocal models near the Turing bifurcation threshold. The analytical prediction agrees with numerical simulation near the Turing bifurcation threshold. Moreover, the analytical and numerical results fit each other well even more remote from the Turing bifurcation threshold for the small values of the nonlocal parameter but not for the higher values. 
\end{abstract}

\begin{keyword}
Nonlocal model \sep Kernel function \sep Turing bifurcation \sep
Spatial-Hopf bifurcation \sep Amplitude equations \sep Weakly nonlinear analysis
\end{keyword}

\end{frontmatter}

\section{Introduction}

Ordinary differential equation models of interacting populations assume their homogeneous distribution over their habitat \cite{kot2001}. However, in reality, the distribution of individuals of different species is heterogeneous and guided by the uneven distribution of favourable resources. A reaction-diffusion model of an interacting population can capture the heterogeneous distribution of the individuals of constituent species and their random movement within their habitat. The main focus of the study with reaction-diffusion models is to understand the species' self-organized distribution (patterns) due to the intra- and inter-specific interaction for favourable resources and survival. Non-homogeneous stationary distributions are classified as spatial self-organization patterns, and time-dependent heterogeneous distributions are known as dynamic patterns. 

In the ecological systems, populations interact with the other individuals at their spatial location and in nearby locations. The usual reaction-diffusion system can not capture such type of phenomena; however, a reaction-diffusion system with nonlocal interaction can. Different type of nonlocal models have been studied for prey-predator interactions in different circumstances, e.g., nonlocal dispersal \cite{sherratt2014}, nonlocal consumption of resources \cite{banerjee2016}, nonlocal intraspecific competition \cite{pal2018,bayliss2017,pal2021,pal2019}, etc. The use of reaction-diffusion equations with nonlocal interaction terms is a newly emerged area of research, whereas this modelling approach is not limited by the models of population interactions. Rather this approach is well accepted for the models in biomedical applications, the study of various diseases, nanotechnology and neuroscience \cite{review,pal2021ad,eftimie2009,sytnyk2021,v2014,paquin2020,pal2021iwbbio}. 

One of our motivations in this paper is derived from the fact that resource-consumer interactions are one of the main building blocks of several food webs and food chains. The dynamics exhibited by the resource-consumer models capture the complex interaction between several trophic levels in natural ecosystems. These models include the dynamical relationships between autotroph-herbivore, prey-predator, host-parasites, etc. According to Abrams \cite{murdoch2013}, `{\it eating is a necessity for all heterotrophic organisms when the foods are themselves living organisms, the interaction between consumer and food is predation}'. A wide variety of mathematical models are available in the literature for prey-predator interactions, but essentially they can be classified based on the Lotka-Volterra type model, Gause type model, and Kolmogorov model \cite{freedman1980deterministic}. The predator population exerts negative feedback on the prey growth through predation, and at the same time, the prey population contributes to the development of predators by supplying energy through food, which helps in the production of new off-springs \cite{fryxell2014}. These two mechanisms are modelled through the functional and numerical responses, respectively. In literature, we find mainly two types of functional responses, namely prey-dependent and predator-dependent \cite{abrams2000nature,lu2022}. For prey-dependent functional response, the rate and amount of consumption of the prey biomass by their predators depend solely upon the prey population density. On the other hand, when the functional response depends upon both the prey and predator population densities, it is called predator-dependent functional response without any ambiguity. However, a numerical response is measured through scaling of the functional response. Various types of biological mechanisms are responsible for the influence of predator population density to shape the functional response.

Mutual interference among the predators, difficulty in getting food at a low prey-to-predator ratio, forming groups to enhance the success in catching and handling large prey – are some well-known mechanisms responsible for the inclusion of predator population density in the functional response. Hunting cooperation describes the cooperative mechanism among the predator individuals to have success in searching, catching, and handling prey \cite{alves2017hunting}. The incorporation of hunting cooperation in the modelling approach results in the inclusion of predator population density in functional and numerical response terms, and hence the functional response is predator-dependent. As a result, the dynamics described by the model become a little bit more complex compared to its counterpart with the prey-dependent functional response. For prey-predator models with specialist predators, the hunting cooperation among the predators can enhance the success in catching prey; however, it exploits the resource excessively, resulting in predator extinction due to scarcity of food at high population density. This negative feedback in the growth of specialist predators is known as the component Allee effect in predators.

Two types of spatial patterns occur in a reaction-diffusion model: stationary and non-stationary. There are several mechanisms for generating various non-stationary or dynamic patterns. However, the existence of a stationary pattern is related to Turing instability. This condition gives a sufficient analytical prerequisite for the existence of the stationary pattern formation \cite{amturing,jdmurray}. The challenging issue is to obtain the analytical prediction about the types of the stationary patterns for the parameter values away from the Turing bifurcation thresholds. On the other hand, near the Turing bifurcation threshold, a finite number of Fourier modes can capture such type of pattern forming scenarios for periodic boundary conditions over a spatial domain \cite{Price,Price1994}. In this case, the solution of the partial differential equation (PDE) model can be approximated by a solution of a system of ordinary differential equations, called amplitude equations. In 1965, Eckhaus used it for the first time to reduce the Navier-Stokes equations of fluid mechanics to a system of ordinary differential equations for the amplitudes \cite{Eckhaus}. This reduction technique has been used in the theory of nonlinear wave interactions to predict the patterns of reaction-diffusion equations. Later on, Segal and Levin derived the amplitude equations for prey-predator systems by the technique known as the weakly nonlinear analysis \cite{segel1976}. This result is a normal form of the spatio-temporal model near the Turing bifurcation threshold.

In this work, we use a basic form of hunting cooperation in prey-predator interaction for the temporal dynamics \cite{alves2017hunting}. We study the existence of the non-trivial equilibrium points of the model and also their stability behaviours through different temporal bifurcations. The most popular weakly nonlinear analysis helps predict the underlying stationary Turing patterns analytically for the local and nonlocal reaction-diffusion models. In this work, we analyze the nonlocal interaction in intraspecific prey competition. For the kernel function, we restrict our attention to  ``thin-tailed'' (Gaussian) function \cite{sherratt2014}. To the best of our knowledge, the amplitude equation for the spatio-temporal models with nonlocal interactions remains unexplored and it is the primary focus of this work. In general, the weakly nonlinear analysis is a multiscale analysis with respect to a bifurcation parameter and the time derivative. In this method, the bifurcation parameter and the time derivative are expanded in terms of a small parameter \cite{zhang2012,ipsen2000,gunaratne1994,li2018}.

\section{Mathematical Model}

For the time $t>0$, suppose $u(t)$ and $v(t)$ are the prey and predator populations, respectively. Following \cite{alves2017hunting}, we consider the simplest form of a prey-predator interaction with hunting cooperation among the predators as:
\begin{subequations}\label{model1}
\begin{align}
\frac{du}{dt} &= \eta u\left(1-\frac{u}{\kappa}\right)-(1+\alpha v)uv \equiv N_{1}(u,v),\\
\frac{dv}{dt} &= (1+\alpha v)uv-v \equiv N_{2}(u,v),
\end{align}
\end{subequations}
with non-negative initial conditions. Here, $\eta$ is the per capita intrinsic growth rate and $\kappa$ is the environmental carrying capacity. The parameter $\alpha$ represents the predator hunting cooperation, and $a=0$ corresponds to the prey-predator model without hunting cooperation \cite{alves2017hunting}. All the parameters involved in the model are assumed to be positive.

We extend the temporal model (\ref{model1}) into the spatio-temporal model to account for the random movements of the species in a two-dimensional habitat. In the mathematical model, the diffusion term captures such types of movements phenomena, and it has been widely accepted in the spatio-temporal prey-predator model. Along with this, researchers have been studied different type of nonlocal interactions in the prey-predator model \cite{sherratt2014,banerjee2016,pal2018,bayliss2017,pal2021}. Following \cite{pal2018}, we consider the nonlocal interaction in the intraspecific competition of the prey population, and in this case, the corresponding integro-differential reaction-diffusion model is given by
\begin{subequations}{\label{NLM}}
\begin{align}
\frac{\partial u}{\partial t} & =  \Delta u +\eta u\left(1-\frac{\psi_{\sigma} \ast u}{\kappa}\right)-(1+\alpha v)uv, \label{NLMa}\\
\frac{\partial v}{\partial t} & = d \Delta v +(1+\alpha v)uv-v, \label{NLMb}
\end{align}
\end{subequations}
with non-negative initial conditions and periodic boundary conditions. We have chosen the non-dimensional diffusion coefficients for the prey and predator species to be $1$ and $d$, respectively. The convolution term $\psi_{\sigma} \ast u$ is defined as
$$
 (\psi_{\sigma} \ast u )(x,y,t) = \int_{-\infty}^{\infty}\int_{-\infty}^{\infty}\psi_{\sigma} (x-w,y-z)u(w,z,t)dwdz,
$$
where $\psi_{\sigma}$ is the kernel function. The parameter $\sigma$ has a prominent role in the nonlocal interactions, it captures the effective area of nonlocal interactions for the nonlocal model (\ref{NLM}). We assume that the kernel function is non-negative, even, normalized and exponentially bounded in $\mathbb{R}^{2}$. Different type of kernel functions are available in the literature \cite{sherratt2014,pal2018,bayliss2017}, but, in this work, we restrict our attention to the Gaussian kernel because it is the most widely used kernel functions in ecological models \cite{sherratt2014,lutscher2005}. It is given by
\begin{equation}{\label{PKF}}
 \psi_{\sigma}(x,y)  = \frac{1}{2\pi\sigma^{2}}e^{-\frac{x^{2}+y^{2}}{2\sigma^{2}}}.
\end{equation}
This specific choice of the kernel function satisfies all the above-mentioned assumptions. Now, we first analyze the temporal dynamics of the model, and then we move toward the local and nonlocal models.

\section{Analysis of the temporal model}

In this section, we study the non-spatial model (\ref{model1}). For analyzing the temporal model, we generally study the equilibrium points and their stabilities, which forwards to the model's bifurcation analysis. Now, the model (\ref{model1}) admits a trivial equilibrium point $E_{0}=(0,0)$ and an axial equilibrium point $E_{1}=(\kappa,0)$. Furthermore, the coexisting equilibrium point (or points) of the system (\ref{model1}) is (are) the point (or points) of the intersection of the non-trivial nullclines
\begin{eqnarray}
u = \frac{\kappa}{\eta}(\eta-(1+\alpha v)v) \equiv n_{1}(v)~\mbox{and}~u =\frac{1}{1+\alpha v} \equiv n_{2}(v),
\end{eqnarray}
inside the first quadrant of the $uv$-plane.
Suppose, $E_{*} = (u_{*},v_{*})$ denotes the components of the coexisting equilibrium point, then $u_{*}(>0)$ satisfies the polynomial equation
\begin{eqnarray}\label{cubiceqn}
\phi(u)\equiv\alpha\eta u^3-\alpha\kappa\eta
u^2-\kappa u+\kappa=0
\end{eqnarray}
and $v_*$ satisfies
\begin{equation}{\label{VSE}}
    v_*=\frac{1}{\alpha}\left(\frac{1}{u_*}-1\right).
\end{equation}
Note that $u_{*}$ has to be less than $1$ to satisfy the positivity of $v_{*}$ (see eq. (\ref{VSE})). The non-trivial prey nullcline $u = n_{1}(v)$, lying in the first quadrant of the $uv$-plane, is a monotone decreasing function in $v$, concave with respect to $u$-axis, and passes through the point $(\kappa,0)$ in the $uv$-plane. On the other hand, the non-trivial predator nullcline $u = n_{2}(v)$, lying in the first quadrant, is a segment of the hyperbola, passes through the point $(1,0)$ in the $uv$-plane. The geometry of these two nullclines ensure the existence of the unique coexisting equilibrium point $E_{*}$ for $\kappa > 1$, e.g., see Fig. \ref{fig:nullclines}(\subref{fig:nullclinesa}). Furthermore, two coexisting equilibrium points exist for $\kappa < 1$, and in this case, two nullclines intersect each other at two different points in the first quadrant of the $uv$-plane, e.g., see Fig. \ref{fig:nullclines}(\subref{fig:nullclinesb}). These points of intersections can not be determined explicitly because of the involvement of the cubic polynomial equation (\ref{cubiceqn}), and they also depend on the magnitudes of other parameters. 

As we have mentioned earlier, all the parameters involved in the cubic equation (\ref{cubiceqn}) are positive. From the Descartes rule of sign, the equation (\ref{cubiceqn}) possesses at most two positive roots, and the feasibility of coexisting equilibrium point demands $u_{*} < 1$. Depending on the parameter values, the coexisting equilibrium point can bifurcate from the axial equilibrium point or can be generated through saddle-node bifurcation. We first present the analytical conditions for the existence and the stability behaviours of these two bifurcations and verify them numerically later on. In most cases, the analytical conditions are implicit in nature.

\begin{figure}[H]
\begin{center}
        \begin{subfigure}[p]{0.45\textwidth}
                \centering
                \includegraphics[width=\textwidth]{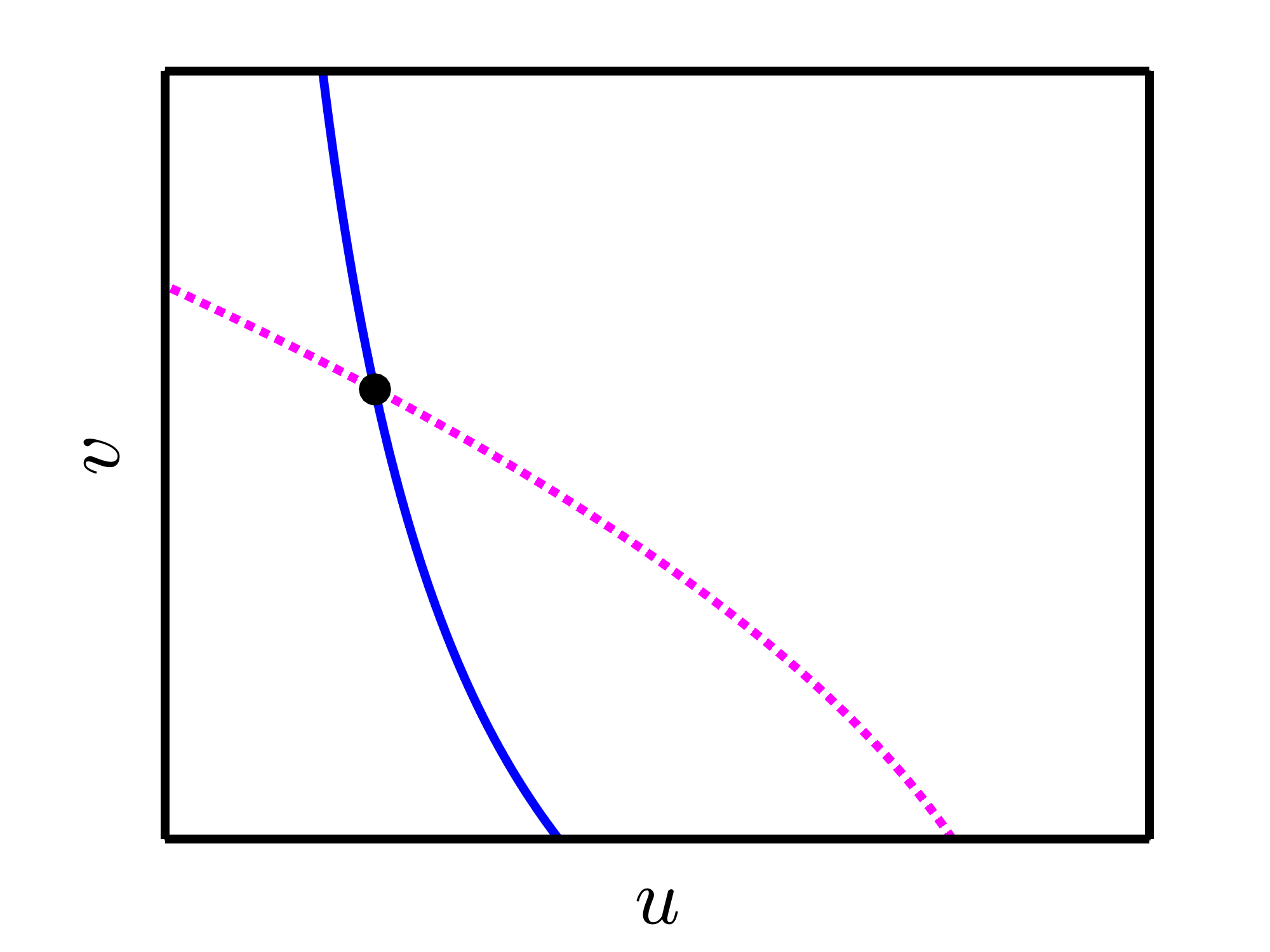}
                \caption{ }\label{fig:nullclinesa}
        \end{subfigure}%
        \begin{subfigure}[p]{0.45\textwidth}
                \centering
                \includegraphics[width=\textwidth]{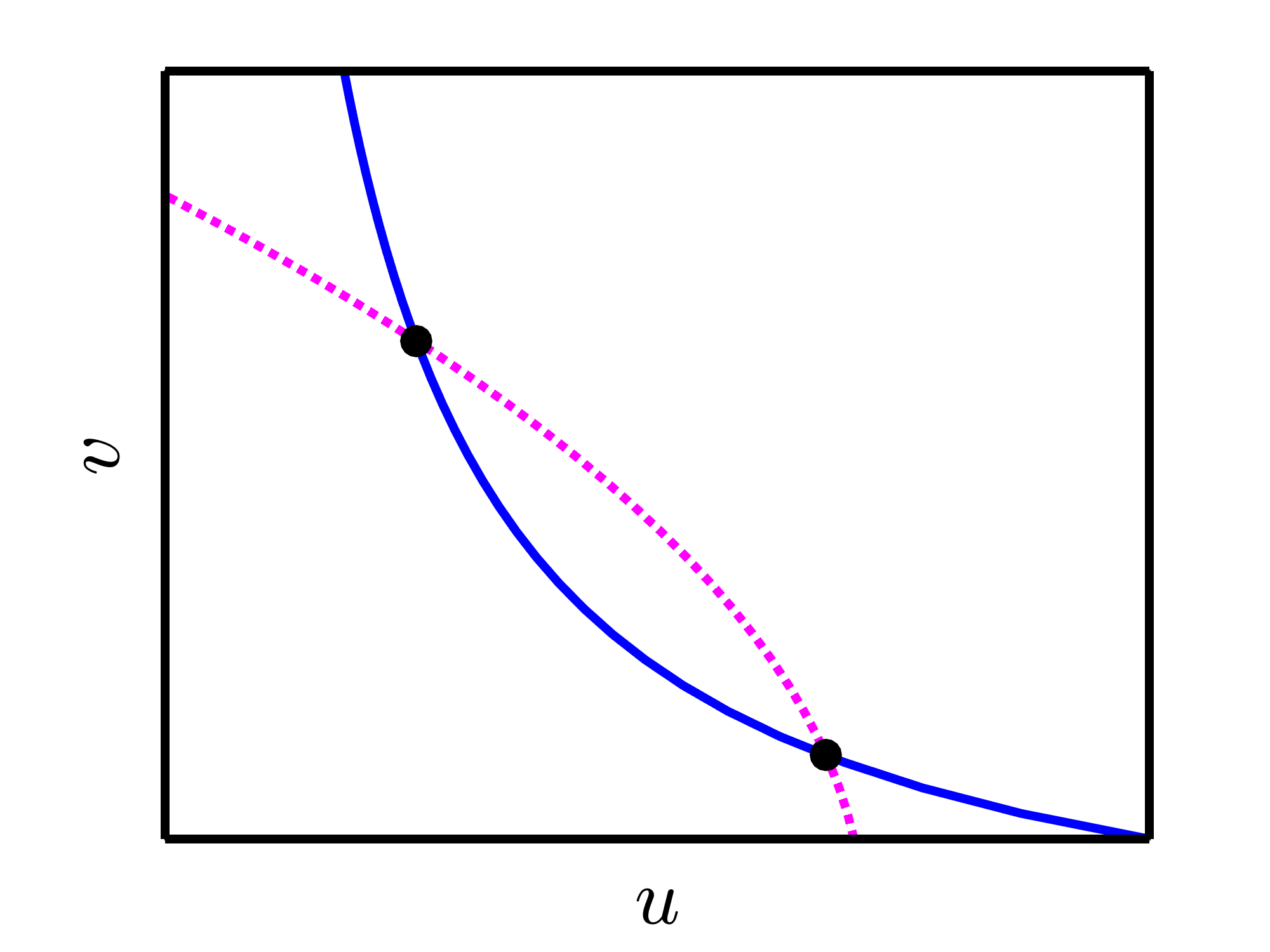}
                \caption{ }\label{fig:nullclinesb}
        \end{subfigure}%
\end{center}
\caption{(Color online) Two different scenarios of the non-trivial nullclines for the system (1). The dotted magenta color curve represents the prey nullcline, and the solid blue color curve represents the predator nullcline. }\label{fig:nullclines}
\end{figure}

Before moving towards the other equilibrium points, we first summarize the stability behaviours of the trivial equilibrium point. With the help of linear stability analysis, we can obtain the local asymptotic stability of all the equilibrium points of the system (\ref{model1}). As we can see, the trivial equilibrium point $E_{0}$ is always unstable, and in particular, it is a saddle point having unstable and stable sub-spaces along the $u$-axis and $v$-axis, respectively. Next, we summarize the stability behaviour of the predator-free equilibrium point $E_{1}$ in the following theorem.

\begin{theorem}
The system (\ref{model1}) undergoes a transcritical bifurcation at $\kappa_{TC}=1$. The equilibrium point $E_{1}$ is stable for $\kappa <1$ and is unstable for $\kappa >1$.
\end{theorem}
\begin{proof}
The Jacobian matrix for the system (\ref{model1}), evaluated at $E_{1} = (\kappa,0)$, is given by
\begin{eqnarray*}
\textbf{J}_1=\left[\begin{array}{cc}
-\eta & -\kappa \\
0 & -1+\kappa \\
\end{array}\right].
\end{eqnarray*}
Therefore, the equilibrium point $E_{1}$ is stable for $\kappa <1$ and unstable for $\kappa >1$. Furthermore, the matrix $\textbf{J}_1$ has a simple zero-eigenvalue for $\kappa = 1$. Now, at $\kappa = 1$, the eigenvectors of $\textbf{J}_{1}$ and $\textbf{J}_{1}^{T}$ associated with the simple zero-eigenvalues are $z = (1,-\eta)^{T}$ and $w = (0, 1)^{T}$, respectively. Finally, we obtain the transversality conditions for the transcritical bifurcation \cite{Perko} at $\kappa = \kappa_{TC}$, and they are given by
$$w^T\mathbf{F}_\kappa(E_1,\kappa_{TC})=0, $$ 
$$w^T\left[D\mathbf{F}_\kappa(E_1,\kappa_{TC})z\right]=0, $$ 
$$w^T\left[D^2\mathbf{F}(E_1,\kappa_{TC})(z,z)\right]=2\eta(\alpha\eta-1)\neq 0, $$
where $\mathbf{F} = (N_{1},N_{2})^{T}$. In particular, the transcritical bifurcation is degenerate as $w^T\left[D\mathbf{F}_\kappa(E_1,\kappa_{TC})z\right]=0$.
\end{proof}

\begin{theorem}
The system (\ref{model1}) undergoes a saddle-node bifurcation at $\eta_{SN}$ for $u_{*}<1$, and the bifurcation threshold is defined implicitly as
\begin{equation*}
    \eta_{SN} = \frac{\kappa (2-u_{*})}{\alpha u_{*}^{3}}.
\end{equation*}
\end{theorem}
\begin{proof}
As we have discussed earlier, two interior equilibrium points can be generated through a saddle-node bifurcation when two non-trivial nullclines $f(u,v)=\eta (1-u/\kappa)-(1+\alpha v)v$ and $g(u,v)=(1+\alpha v)u-1$ touch (excluding the intersecting case) each other at $E_{*} = (u_{*},v_{*})$ in the first quadrant. In this case, at $E_{*}$, both the nullclines share a common tangent. Therefore, at $E_{*}$, we must have
\begin{equation}{\label{CCT}}
    \left.-\frac{f_u}{f_v}\right|_{E_{*}} = \left.-\frac{g_u}{g_v}\right|_{E_{*}}.
\end{equation}
Now, the Jacobian matrix for the system (\ref{model1}) at $E_{*}$ is given by
\begin{eqnarray*}
\textbf{J}_{*}=\left[\begin{array}{cc}
uf_{u} & uf_{v} \\
vg_{u} & vg_{v} \\
\end{array}\right]_{E_*}.
\end{eqnarray*}
By using (\ref{CCT}), we obtain $\det(\textbf{J}_{*}) = 0$ and $\textrm{tr}(J_*)\,=\,u_{*}(\alpha v_{*}-\eta /\kappa) \neq 0$. Therefore, the matrix $\textbf{J}_{*}$ has a zero eigenvalue with multiplicity one. If we solve $\det(\textbf{J}_{*}) = 0$ implicitly for $\eta$, we can find the threshold for $\eta_{SN}$ by which the saddle-node bifurcation exists. To verify the transversality conditions for the saddle-node bifurcation, we consider $\eta$ as the bifurcation parameter denoted by $\eta_{SN}$. The saddle-node bifurcation threshold $\eta_{SN}$ is a positive root of the implicit equation $\alpha\eta u_*=\kappa(1+\alpha v_*)(1+2\alpha v_*)$. Hence, at $\eta=\eta_{SN}$, the matrix $\textbf{J}_*$ has a simple zero eigenvalue, while its eigenvectors, associated with zero eigenvalue of the matrices $\textbf{J}_{*}$ and $\textbf{J}_{*}^{T}$, are 
\begin{eqnarray*}
z=\left[\begin{array}{c}
1 \\
-\frac{1+\alpha v_*}{\alpha u_*} \\
\end{array}\right]\,\,\textrm{and}\,\,
w=\left[\begin{array}{c}
1 \\
\frac{1+2\alpha v_*}{\alpha v_*} \\
\end{array}\right],
\end{eqnarray*}
respectively. The transversality conditions for the saddle-node bifurcation are as follows:
$$w^t\mathbf{F}_\eta(E_*,\eta_{SN})\,=\,u_*\left(1-\frac{u_*}{\kappa}\right)\,\neq\, 0, $$ 
$$w^t\left[D^2\mathbf{F}(E_*,\eta_{SN})(z,z)\right]\,=\,-\frac{2\eta}{\kappa}-\frac{2}{\alpha u_*^2}(1+\alpha v_*)\,\neq\, 0. $$
\end{proof}

The considered model admits either one or two interior equilibrium points, but none of the components can be determined explicitly in either of these cases. This prevents us from finding the exact condition for local asymptotic stability of the coexisting equilibrium point. However, we summarize the local asymptotic stability condition for a typical coexisting equilibrium point $E_{*}=(u_{*},v_{*})$ in an implicit way in the following theorem.

\begin{theorem}
The coexisting equilibrium point $E_{*} = (u_{*},v_{*})$ of the system (\ref{model1}) is locally asymptotically stable if $\alpha > \kappa v_{*}/\eta$ and $\kappa (1+\alpha v_{*})(1+2\alpha v_{*})>\eta\alpha u_{*}$ with $u_{*} < 1$.
\end{theorem}
\begin{proof}
For the feasibility of $v_{*}$, we assume $u_{*} < 1$. The Jacobian matrix for (\ref{model1}) at $E_{*}$ is given by
\begin{equation}\label{jacob1}
\textbf{J}_{*}=\begin{bmatrix}
 -\frac{\eta u_{*}}{\kappa} & -(1+2\alpha v_{*})u_{*} \\
(1+\alpha v_{*})v_{*} & \alpha u_{*}v_{*}
\end{bmatrix}.
\end{equation}
Therefore, the equilibrium point $E_{*}$ is locally asymptotically stable if $\mbox{tr}(\textbf{J}_{*})<0$ and $\det(\textbf{J}_{*})>0$ hold, i.e., $\alpha > \kappa v_{*}/\eta$ and $\kappa(1+\alpha v_{*})(1+2\alpha v_{*})>\eta\alpha u_{*}$. 
\end{proof}

Depending on the parameter values, sometimes the implicit condition $\kappa(1+\alpha v_*)(1+2\alpha v_*)>\eta\alpha u_*$ holds, but the other condition $\alpha>\kappa v_*/\eta$ is violated, and in this case, the coexisting equilibrium point loses its stability through the Hopf bifurcation. We summarize all the conditions for the Hopf bifurcation in the following theorem.

\begin{theorem}
The coexisting equilibrium point $E_{*} = (u_{*},v_{*})$ with $u_{*} < 1$ undergoes a Hopf bifurcation when $\alpha$ crosses the implicit threshold $\alpha_{H} = \kappa v_{*}/\eta$, while maintaining the inequality $\kappa(1+\alpha v_{*})(1+2\alpha v_{*})>\eta\alpha u_{*}$.
\end{theorem}
\begin{proof}
The Jacobian matrix for (\ref{model1}) at $E_{*}$ is given in (\ref{jacob1}) and the trace of the matrix $\textbf{J}_{*}$ is equal to $(\alpha v_{*}-\eta/\kappa)u_{*}$. The trace of the Jacobian matrix $\textbf{J}_{*}$ equated to zero gives the Hopf bifurcation threshold in terms of $\alpha$ implicitly as $\alpha_{H} = \kappa v_{*}/\eta$. The other condition $\kappa(1+\alpha v_{*})(1+2\alpha v_{*})>\eta\alpha u_{*}$ ensures that the determinant of the Jacobian matrix $\textbf{J}_{*}$ is positive. The transversality condition for the Hopf bifurcation is given by
\begin{equation}
    \left.\frac{d}{d\alpha}(\textrm{Re}(\lambda))\right|_{E_{*};\alpha_{H}} \neq 0,
\end{equation} where $\lambda$ is a root of the Jacobian matrix $\textbf{J}_{*}$. 
\end{proof}

Until now, we have explained the generation of coexistence equilibrium points through saddle-node bifurcation and their destabilization through Hopf bifurcation. Due to the implicit involvement of the non-trivial equilibrium point with the parameters, we have chosen $\eta$ and $\alpha$ as bifurcation parameters for finding the bifurcation thresholds. We can also find saddle-node and Hopf bifurcation curves in the two-dimensional parameter plane by considering $\eta$ and $\alpha$ as bifurcation parameters. Therefore, it is expected that the two local bifurcation curves will intersect at a Bogdanov-Takens bifurcation point of co-dimension two. 

\begin{theorem} The coexisting equilibrium point $E_*$ undergoes a Bogdanov-Takens bifurcation when the implicit parametric conditions $\alpha_{BT}=\kappa v_*/\eta$ and $\eta_{BT} = \kappa (2-u_{*})/\alpha u_{*}^3$ are satisfied.
\end{theorem}
\begin{proof} At $E_*$, the trace and determinant of $\textbf{J}_{*}$ are given by 
$$\textrm{tr}(\textbf{J}_{*}) = -\frac{\eta u_{*}}{\kappa}+\alpha u_{*}v_{*},\,\,\textrm{det}(\textbf{J}_{*})\,=\,u_*v_*\left[(1+\alpha v_*)(1+2\alpha v_*)-\frac{\eta\alpha u_*}{\kappa}\right].$$
The Bogdanov-Takens bifurcation can be obtained by solving $\textrm{tr}(\textbf{J}_{*}) = 0$ and $\textrm{det}(\textbf{J}_{*}) = 0$ simultaneously for $\eta$ and $\alpha$ with the facts $u_*,v_*>0$ and $1+\alpha v_*=1/u_{*}$. Therefore, from these conditions, we find
$$\alpha_{BT}=\frac{\kappa v_*}{\eta},\,\,\eta_{BT}\,=\,\frac{\kappa (2-u_*)}{\alpha u_*^3}.$$
This bifurcation ensures that the matrix $\textbf{J}_{*}$ has zero as an eigenvalue with multiplicity two. Moreover, satisfying all the parametric restrictions mentioned above, the Jordan canonical form of $\textbf{J}_{*}$ at $(\alpha,\,\eta)\,=\,(\alpha_{BT},\,\eta_{BT})$ can be obtained as $\left[\begin{array}{cc}
0 & 1 \\
0 & 0 \\
\end{array}\right]$. 
\end{proof}

\section{Analysis of the nonlocal spatio-temporal model}{\label{SE2}}

In this section, we first find the Turing bifurcation conditions for the local and nonlocal models and then move towards the weakly nonlinear analysis for both models. Note that the equilibrium points of the temporal model are the homogeneous solutions for both local and nonlocal models. This invariance of homogeneous steady-state is due to the choice of the kernel function and the boundary conditions.

\subsection{Linear stability analysis}
We first assume that the homogeneous steady-state $E_{*} = (u_{*},v_{*})$ of the system (\ref{model1}) is locally asymptotically stable, i.e., $\mbox{tr}(\mathbf{J}_{*})<0$ and $\mbox{det}(\mathbf{J}_{*})>0$. Now, we perturb the homogeneous steady-state by $u=u_{*}+\epsilon \tilde{u}_{1} \exp{(\lambda t+i(k_{x}x+k_{y}y))}$ and $v=v_{*}+\epsilon \tilde{v}_{1}\exp{(\lambda t+i(k_{x}x+k_{y}y))}$, where $|\epsilon|\ll 1$. Substitution of $u$ and $v$ into the system (\ref{NLM}) and linearization leads to
\begin{equation}{\label{LFTM}}
\begin{bmatrix}
 a_{11}- (\eta/\kappa)u_{*}\widehat{\psi}_\sigma(k_{x},k_{y}) -(k_{x}^{2}+k_{y}^{2}) -\lambda & a_{12} \\
 a_{21} & a_{22}-d(k_{x}^{2}+k_{y}^{2})-\lambda
\end{bmatrix} \begin{bmatrix}
 \tilde{u}_{1}\\
 \tilde{v}_{1}
\end{bmatrix} \equiv \mathbf{M}\begin{bmatrix}
 \tilde{u}_{1}\\
 \tilde{v}_{1}
\end{bmatrix} = \begin{bmatrix}
 0\\
 0
\end{bmatrix} ,
\end{equation}
where $a_{11}=0$, $a_{12}=-(1+2\alpha v_{*})u_{*}$, $a_{21}= (1+\alpha v_{*})v_{*}$, $a_{22}=\alpha u_{*}v_{*}$, and $\widehat{\psi}_\sigma(k_{x},k_{y}) =\mbox{exp}(-\sigma^{2}(k_{x}^{2}+k_{y}^{2})/2)$ is the Fourier transform of the kernel function $\psi_{\sigma}(x,y)$ in two variables. For the non-trivial solution of the matrix equation (\ref{LFTM}), the determinant of the matrix $\mathbf{M}$ has to be equal to $0$, and it leads to the characteristic equation
\begin{equation}{\label{CE1}}
\lambda^{2}-\mathcal{T}(k)\lambda+\mathcal{D}(k) = 0,
\end{equation} where 
$\mathcal{T}(k) =  a_{11} -(\eta/\kappa)u_{*}\exp{(-\sigma^{2}k^{2}/2)} +a_{22} -(1+d)k^{2}$ and $\mathcal{D}(k) = ( a_{11}-(\eta/\kappa)u_{*}\exp{(-\sigma^{2}k^{2}/2)} -k^{2})(a_{22}-dk^{2}) -a_{12}a_{21}$ with $k^{2}=k_{x}^{2}+k_{y}^{2}$ with $k$ being the wave number.

From (\ref{CE1}), we obtain
\begin{equation}{\label{EVNL}}
\lambda_{\pm}(k)=\frac{\mathcal{T}(k) \pm \sqrt{(\mathcal{T}(k))^{2}-4\mathcal{D}(k)}}{2}.
\end{equation}
For a fixed $\sigma$, we see that $\mathcal{T}(k) <0$ holds for all $k>0$ and $d$, as $\mbox{tr}(\mathbf{J}_{*})<0$, but we can not conclude anything about the sign of $\mathcal{D}(k)$. If $\mathcal{D}(k) > 0$ holds for all $k>0$ at given values of $d$ and $\sigma$, then the homogeneous steady-state $(u_{*},v_{*})$ is stable under the heterogeneous perturbations. Violating the condition $\mathcal{D}(k) > 0$ for some $k>0$ causes an instability in the homogeneous steady-state $(u_{*},v_{*})$, called Turing instability. In this case, a Turing pattern can be observed in the system (\ref{NLM}) for such parameter values. 

So, for a fixed value of $\sigma$, we focus on finding the critical value $d=d_{T}$ and the corresponding critical value $k=k_{T} >0$ for which the Turing bifurcation occurs. At the Turing bifurcation threshold, $\mathcal{D}(k) = 0$ holds for a unique $k$, and in this case $\mathcal{D}(k)$ satisfies
\begin{equation}{\label{NLTBC}}
\mathcal{D}(k) = 0 ~\mbox{and}~\frac{\partial\mathcal{D}(k) }{\partial k}=0.
\end{equation}
From the first equation of (\ref{NLTBC}), we get
\begin{equation}{\label{TBCD}}
d=\frac{a_{22}(a_{11}-u_{*}e^{-k^{2}\sigma^{2}/2}-k^{2})- a_{12}a_{21}}{k^{2}(a_{11}-u_{*}e^{-k^{2}\sigma^{2}/2}-k^{2})}.
\end{equation}
After eliminating $d$ between the first and second equations of (\ref{NLTBC}), we obtain
\begin{align}{\label{TBCK}}
2a_{22}(a_{11}-u_{*}e^{-\sigma^{2}k^{2}/2} -k^{2})^{2}-a_{12}a_{21}\bigg{(}2a_{11}-4k^{2}-u_{*}(2-k^{2}\sigma^{2})e^{-k^{2}\sigma^{2}/2}\bigg{)}=0.
\end{align}
For a fixed $\sigma$, we find the solution $k=k_{T}$ of the equation (\ref{TBCK}) numerically, By substituting it into the equation (\ref{TBCD}), we find the critical diffusion coefficient $d_{T}$ for the Turing bifurcation. The Turing bifurcation threshold for the local model can be obtained by calculating $d=d_{T}$ with $\sigma =0$.

\subsection{Weakly Nonlinear Analysis}

Near the Turing bifurcation threshold, the dynamics of the system (\ref{NLM}) change slowly due to the small variation in the parameter values, and we can study the pattern formation for such parameters with the help of the amplitude equations. In this case, we consider three active dominant resonant pairs of eigenmodes $(\mathbf{k}_{j},-\mathbf{k}_{j})~(j=1,2,3)$ making angles of $2\pi/3$ with $|\mathbf{k}_{j}|=k_{T}$. The solution of the nonlocal model (\ref{NLM}) near the Turing bifurcation threshold $d = d_{T}$ can be written as the sum of the Fourier plane waves:
\begin{equation}{\label{SONLM}}
\begin{pmatrix}
u \\
v
\end{pmatrix} = \begin{pmatrix}
u_{*} \\
v_{*}
\end{pmatrix} + \sum_{j=1}^{3} [\textbf{A}_{j}(t) \exp (i\mathbf{k}_{j}\cdot \mathbf{r}) + \overline{\textbf{A}}_{j}(t)\exp (-i\mathbf{k}_{j}\cdot \mathbf{r})],
\end{equation}
where $\textbf{A}_{j} = (A_{j,u},A_{j,v})^{T}$ and $\overline{\textbf{A}}_{j} = (\overline{A}_{j,u},\overline{A}_{j,v})^{T}$ are the amplitudes associated with the eigenmodes $\mathbf{k}_{j}$ and $-\mathbf{k}_{j}$, respectively.

Now, we derive the amplitude equations for $\textbf{A}_{j} (j=1,2,3)$. We expand the time in terms of different time-scales with respect to a small parameter $\epsilon$ and also expand the bifurcation parameter $d$ and the other variables $u$ and $v$ as follows:
\begin{subequations}{\label{PST}}
\begin{align}
t &= t_{0} + \epsilon t_{1} + \epsilon^{2}t_{2} + \cdots,\\
d &= d_{T} + \epsilon d^{(1)} + \epsilon^{2}d^{(2)} + \cdots, \\
u &= u_{*} + \epsilon u_{1} + \epsilon^{2}u_{2} + \epsilon^{3}u_{3} + \cdots, \\
v &= v_{*} + \epsilon v_{1} + \epsilon^{2}v_{2} + \epsilon^{3}v_{3} + \cdots.
\end{align}
\end{subequations}

The amplitude $\textbf{A}_{j}$ of the spatial pattern evolves on a slow temporal scale at the initial stage. The derivative $\frac{\partial}{\partial t_{0}}$ does not have an effect on $\textbf{A}_{j}$ because it corresponds  to the fast time. Therefore, we separate the fast and slow time scales as 
\begin{equation}{\label{SFTS}}
\frac{\partial}{\partial t} =\epsilon\frac{\partial}{\partial t_{1}} + \epsilon^{2}\frac{\partial}{\partial t_{2}} + o(\epsilon^{3}). 
\end{equation}
We substitute (\ref{PST}) and (\ref{SFTS}) into (\ref{NLM}) and equate the coefficients of $\epsilon$, $\epsilon^{2}$ and $\epsilon^{3}$. Comparing the first order coefficients of $\epsilon$, we obtain
\begin{equation}{\label{COFOE}}
\mathbf{L}_{T} \begin{pmatrix}
u_{1} \\
v_{1}
\end{pmatrix} \equiv \begin{pmatrix}
f_{10} +\Delta-(\eta/\kappa)u_{*}\psi_{\sigma}\ast & f_{01} \\
g_{10} & g_{01} +d_{T}\Delta 
\end{pmatrix}\begin{pmatrix}
u_{1} \\
v_{1}
\end{pmatrix} = \begin{pmatrix}
0 \\
0
\end{pmatrix},
\end{equation}
where $f_{10}=a_{11}$, $f_{01} = a_{12}$, $g_{10}= a_{21}$, and $g_{01}=a_{22}$.

The solution of the system (\ref{COFOE}) can be written in the form 
\begin{equation}{\label{COFOS}}
\begin{pmatrix}
u_{1} \\
v_{1}
\end{pmatrix} = \begin{pmatrix}
f_{1} \\
g_{1}
\end{pmatrix} \bigg{(} \sum_{j=1}^{3}W_{j} \mbox{exp}(i\mathbf{k}_{j}\cdot \mathbf{r}) \bigg{)} + \mbox{c.c.},
\end{equation}
where 
$$f_{1} = -f_{01}/f_{c}~\mbox{and}~g_{1}=(f_{10}-(\eta/\kappa)u_{*}e^{-k_{T}^{2}\sigma^{2}/2}-k_{T}^{2})/f_{c}$$ with $f_{c}=\sqrt{(f_{10}-(\eta/\kappa)u_{*}\exp{(-k_{T}^{2}\sigma^{2}/2)}-k_{T}^{2})^2+f_{01}^2}$. Here, for given $j$, $W_{j}$ is the modulus of the first order disturbance term, and c.c. denotes the complex conjugate.

Now, comparing the second order coefficients of $\epsilon$, we get
\begin{align}{\label{COSOE}}
\mathbf{L}_{T} \begin{pmatrix}
u_{2} \\
v_{2}
\end{pmatrix} = \frac{\partial}{\partial t_{1}} & \begin{pmatrix}
u_{1} \\
v_{1}
\end{pmatrix} - \begin{pmatrix}
0 & 0 \\
0 & d^{(1)}\Delta 
\end{pmatrix} \begin{pmatrix}
u_{1} \\
v_{1}
\end{pmatrix}  \nonumber \\
&-\begin{pmatrix}
f_{20}u_{1}^{2} +f_{11}u_{1}v_{1}  + f_{02}v_{1}^{2}-(\eta/\kappa)u_{1}\psi_{\sigma}\ast u_{1} \\
g_{20}u_{1}^{2} +g_{11}u_{1}v_{1} + g_{02} v_{1}^{2} 
\end{pmatrix} = \begin{pmatrix}
F_{u}\\
F_{v}
\end{pmatrix},
\end{align}
where  $f_{20}=0$, $f_{11}=-(1+2\alpha v_{*})$, $f_{02}= -\alpha u_{*}$, $g_{20}= 0 $, $g_{11}= (1+2\alpha v_{*})$, and $g_{02}= \alpha u_{*}$.

Our next target is to solve the system (\ref{COSOE}) for $(u_{2},v_{2})^{T}$. The Fredholm solvability condition ensures the existence of a nontrivial solution of the non-homogeneous problem (\ref{COSOE}). According to that condition, the right-hand side of the equation (\ref{COSOE}) must be orthogonal to the zero eigenvectors of the operator $\mathbf{L}_{T}^{+}$ (the adjoint operator of the operator $\mathbf{L}_{T}$). Now, the zero eigenvector of the operator $\mathbf{L}_{T}$ is
\begin{equation}{\label{FSCSO}}
\begin{pmatrix}
f_{2}\\
g_{2}
\end{pmatrix} \mbox{exp}(-i\mathbf{k}_{j}\cdot \mathbf{r}) + \mbox{c.c.}, 
\end{equation}
where $$f_{2} = -g_{10}/g_{c}~\mbox{and}~ g_{2} = (f_{10}-(\eta/\kappa)u_{*}e^{-k_{T}^{2}\sigma^{2}/2} -k_{T}^{2})/g_{c},$$ with $g_{c} = \sqrt{(f_{10}-(\eta/\kappa)u_{*}\exp{(-k_{T}^{2}\sigma^{2}/2)} -k_{T}^{2})^{2}+g_{10}^{2}}$.

\noindent From the orthogonality condition, we obtain
\begin{equation}
(f_{2},g_{2})\begin{pmatrix}
F_{u}^{j}\\
F_{v}^{j}
\end{pmatrix} = 0,~~ (j=1,2,3),
\end{equation}
where $F_{u}^{j}$ and $F_{v}^{j}$ are the coefficients of $\mbox{exp}(i\mathbf{k}_{j}\cdot \mathbf{r})$ in $F_{u}$ and $F_{v}$ respectively. 

\noindent For $j=1$, we get
\begin{equation}
\begin{pmatrix}
F_{u}^{1}\\
F_{v}^{1}
\end{pmatrix} = \begin{pmatrix}
f_{1} \\
g_{1}
\end{pmatrix} \frac{\partial W_{1}}{\partial t_{1}} + \begin{pmatrix}
0 \\
g_{1}k_{T}^{2}d^{(1)}
\end{pmatrix} W_{1}- 2\begin{pmatrix}
F_{1}\\
G_{1}
\end{pmatrix}\overline{W}_{2}\overline{W}_{3},
\end{equation}
where $F_{1} = f_{20}f_{1}^{2} +f_{11}f_{1}g_{1}+ f_{02}g_{1}^{2}-(\eta/\kappa)f_{1}^{2}\exp{(-k_{T}^{2}\sigma^{2}/2)}$ and $G_{1} =g_{20}f_{1}^{2} +g_{11}f_{1}g_{1}+ g_{02}g_{1}^{2}$. 

\noindent After applying the solvability condition, we get
\begin{equation}
(f_{1}f_{2}+g_{1}g_{2})\frac{\partial W_{1}}{\partial t_{1}} =-g_{1}g_{2}k_{T}^{2}d^{(1)}W_{1}+ 2(f_{2}F_{1}+g_{2}G_{1})\overline{W}_{2}\overline{W}_{3}.
\end{equation}
Similarly, for $j=2$ and $3$, we obtain the following equations:
\begin{equation}
(f_{1}f_{2}+g_{1}g_{2})\frac{\partial W_{2}}{\partial t_{1}} = -g_{1}g_{2}k_{T}^{2}d^{(1)}W_{2}+ 2(f_{2}F_{1}+g_{2}G_{1})\overline{W}_{1}\overline{W}_{3},
\end{equation}
\begin{equation}
(f_{1}f_{2}+g_{1}g_{2})\frac{\partial W_{3}}{\partial t_{1}} =-g_{1}g_{2}k_{T}^{2}d^{(1)}W_{3}+ 2(f_{2}F_{1}+g_{2}G_{1})\overline{W}_{1}\overline{W}_{2}.
\end{equation}

The solution of the system (\ref{COSOE}) can be written as:
\begin{align}{\label{COSOS}}
\begin{pmatrix}
u_{2} \\
v_{2}
\end{pmatrix} = \begin{pmatrix}
X_{0} \\
Y_{0}
\end{pmatrix} & + \sum_{j=1}^{3}\begin{pmatrix}
X_{j} \\
Y_{j}
\end{pmatrix} \mbox{exp}(i\mathbf{k}_{j}\cdot \mathbf{r}) + \sum_{j=1}^{3}\begin{pmatrix}
X_{jj} \\
Y_{jj}
\end{pmatrix} \mbox{exp}(i2\mathbf{k}_{j}\cdot \mathbf{r}) \nonumber \\ 
&+ \begin{pmatrix}
X_{12} \\
Y_{12}
\end{pmatrix} \mbox{exp}(i(\mathbf{k}_{1}-\mathbf{k}_{2})\cdot \mathbf{r}) + \begin{pmatrix}
X_{23} \\
Y_{23}
\end{pmatrix} \mbox{exp}(i(\mathbf{k}_{2}-\mathbf{k}_{3})\cdot \mathbf{r}) \nonumber \\
&+ \begin{pmatrix}
X_{13} \\
Y_{13}
\end{pmatrix} \mbox{exp}(i(\mathbf{k}_{1}-\mathbf{k}_{3})\cdot \mathbf{r}) + \mbox{c.c.}
\end{align}
Substituting (\ref{COSOS}) and (\ref{COFOS}) into (\ref{COSOE}), and collecting the constant terms and the coefficients of $\mbox{exp}(i\mathbf{k}_{j}\cdot \mathbf{r})$, $\mbox{exp}(i2\mathbf{k}_{j}\cdot \mathbf{r})$ and $\mbox{exp}(i(\mathbf{k}_{j}-\mathbf{k}_{m})\cdot \mathbf{r})$ ($1 \leq j,m \leq 3$ and $j\neq m$), we find
\begin{subequations}{\label{UCSOS}}
\begin{align}
\begin{pmatrix}
X_{0} \\
Y_{0}
\end{pmatrix} & = -2\begin{pmatrix}
f_{10}-(\eta/\kappa)u_{*} & f_{01}\\
g_{10} & g_{01}
\end{pmatrix}^{-1} \begin{pmatrix}
F_{1}\\
G_{1}
\end{pmatrix} (|W_{1}|^{2}+|W_{2}|^{2}+|W_{3}|^{2}) \nonumber \\
&= \begin{pmatrix}
\xi_{u0}\\
\xi_{v0}
\end{pmatrix} (|W_{1}|^{2}+|W_{2}|^{2}+|W_{3}|^{2}), \\
\begin{pmatrix}
X_{j} \\
Y_{j}
\end{pmatrix} & = \begin{pmatrix}
f_{1} \\
g_{1}
\end{pmatrix}Z_{j},\\
\begin{pmatrix}
X_{jj} \\
Y_{jj}
\end{pmatrix} & = -\begin{pmatrix}
f_{10}-4k_{T}^{2}-(\eta/\kappa)u_{*}e^{-4k_{T}^{2}\sigma^{2}/2}  & f_{01} \\
g_{10} & g_{01}-4d_{T}k_{T}^{2}
\end{pmatrix}^{-1} \begin{pmatrix}
F_{1}\\
G_{1}
\end{pmatrix} W_{j}^{2}\nonumber \\
&= \begin{pmatrix}
\xi_{u1}\\
\xi_{v1}
\end{pmatrix} W_{j}^{2},\\
\begin{pmatrix}
X_{jm} \\
Y_{jm}
\end{pmatrix} & = -2\begin{pmatrix}
f_{10}-3k_{T}^{2}-(\eta/\kappa)u_{*}e^{-3k_{T}^{2}\sigma^{2}/2}  & f_{01} \\
g_{10} & g_{01}-3d_{T}k_{T}^{2}
\end{pmatrix}^{-1} \begin{pmatrix}
F_{1}\\
G_{1}
\end{pmatrix} W_{j}\overline{W}_{m} \nonumber \\
&= \begin{pmatrix}
\xi_{u2}\\
\xi_{v2}
\end{pmatrix} W_{j}\overline{W}_{m}.
\end{align}
\end{subequations}

\noindent Now, comparing the third order coefficients of $\epsilon$, we obtain
\begin{align}{\label{COTOE}}
\mathbf{L}_{T}  \begin{pmatrix}
u_{3} \\
v_{3}
\end{pmatrix} = & \begin{pmatrix}
\frac{\partial u_{2}}{\partial t_{1}}+\frac{\partial u_{1}}{\partial t_{2}}\\
\frac{\partial v_{2}}{\partial t_{1}}+\frac{\partial v_{1}}{\partial t_{2}}
\end{pmatrix} - \begin{pmatrix}
0 & 0 \\
0 & d^{(2)}\Delta 
\end{pmatrix} \begin{pmatrix}
u_{1} \\
v_{1}
\end{pmatrix} - \begin{pmatrix}
0 & 0 \\
0 & d^{(1)}\Delta 
\end{pmatrix} \begin{pmatrix}
u_{2} \\
v_{2}
\end{pmatrix}  \nonumber \\
&- \begin{pmatrix}
2f_{20}u_{1}u_{2} +f_{11}(u_{1}v_{2}+u_{2}v_{1}) + 2f_{02}v_{1}v_{2}-(\eta/\kappa)u_{1}\psi_{\sigma}\ast u_{2}-(\eta/\kappa)u_{2}\psi_{\sigma}\ast u_{1} \\
2g_{20}u_{1}u_{2} +g_{11}(u_{1}v_{2}+u_{2}v_{1})  + 2g_{02}v_{1}v_{2}  
\end{pmatrix}  \nonumber \\
&~~~  -\begin{pmatrix}
f_{30}u_{1}^{3} +f_{21}u_{1}^{2}v_{1} + f_{12}u_{1}v_{1}^{2} + f_{03}v_{1}^{3} \\
g_{30}u_{1}^{3} +g_{21}u_{1}^{2}v_{1} + g_{12}u_{1}v_{1}^{2} + g_{03} v_{1}^{3}
\end{pmatrix},
\end{align}
where $f_{30}= 0$, $f_{21}= 0$, $f_{12}= -\alpha$, $f_{03}= 0$, $g_{30}= 0$, $g_{21}= 0$, $g_{12}= \alpha$, and $g_{03}= 0$.

After applying the Fredholm solvability condition to the system (\ref{COTOE}) we obtain:
\begin{subequations}{\label{CFSCTO}}
\begin{align}
(f_{1}f_{2}+g_{1}g_{2})\bigg{(} \frac{\partial W_{1}}{\partial t_{2}} + \frac{\partial Z_{1}}{\partial t_{1}} \bigg{)} = & -g_{1}g_{2}k_{T}^{2}(d^{(1)}Z_{1}+d^{(2)}W_{1})+ 2(f_{2}F_{1}+g_{2}G_{1})(\overline{W}_{2}\overline{Z}_{3} + \overline{W}_{3}\overline{Z}_{2}) \nonumber \\
&~~+ ((f_{2}F_{2}+g_{2}G_{2})|W_{1}|^{2} + (f_{2}F_{3}+g_{2}G_{3})(|W_{2}|^{2}+|W_{3}|^{2}))W_{1}, \\
(f_{1}f_{2}+g_{1}g_{2})\bigg{(} \frac{\partial W_{2}}{\partial t_{2}} + \frac{\partial Z_{2}}{\partial t_{1}} \bigg{)} = & -g_{1}g_{2}k_{T}^{2}(d^{(1)}Z_{2}+d^{(2)}W_{2})+ 2(f_{2}F_{1}+g_{2}G_{1})(\overline{W}_{1}\overline{Z}_{3} + \overline{W}_{3}\overline{Z}_{1}) \nonumber \\
&~~+ ((f_{2}F_{2}+g_{2}G_{2})|W_{2}|^{2} + (f_{2}F_{3}+g_{2}G_{3})(|W_{1}|^{2}+|W_{3}|^{2}))W_{2}, \\
(f_{1}f_{2}+g_{1}g_{2})\bigg{(} \frac{\partial W_{3}}{\partial t_{2}} + \frac{\partial Z_{3}}{\partial t_{1}} \bigg{)} = &-g_{1}g_{2}k_{T}^{2}(d^{(1)}Z_{3}+d^{(2)}W_{3})+  2(f_{2}F_{1}+g_{2}G_{1})(\overline{W}_{1}\overline{Z}_{2} + \overline{W}_{2}\overline{Z}_{1}) \nonumber \\
&~~+ ((f_{2}F_{2}+g_{2}G_{2})|W_{3}|^{2} + (f_{2}F_{3}+g_{2}G_{3})(|W_{1}|^{2}+|W_{2}|^{2}))W_{3},
\end{align}
\end{subequations}
where 
\begin{align*}
F_{2} &= 2f_{20}f_{1}(\xi_{u0}+\xi_{u1})+f_{11}(f_{1}(\xi_{v0}+\xi_{v1})+g_{1}(\xi_{u0}+\xi_{u1}))+2f_{02}g_{1}(\xi_{v0}+\xi_{v1})\\
&~~~~~~-(\eta/\kappa)f_{1}(\xi_{u0}+\xi_{u1}e^{-4k_{T}^{2}\sigma^{2}/2})-(\eta/\kappa)f_{1}(\xi_{u0}+\xi_{u1})e^{-k_{T}^{2}\sigma^{2}/2}+3F_{4},\\
G_{2} &= 2g_{20}f_{1}(\xi_{u0}+\xi_{u1})+g_{11}(f_{1}(\xi_{v0}+\xi_{v1})+g_{1}(\xi_{u0}+\xi_{u1}))+2g_{02}g_{1}(\xi_{v0}+\xi_{v1})+3G_{4},\\
F_{3} &= 2f_{20}f_{1}(\xi_{u0}+\xi_{u2})+f_{11}(f_{1}(\xi_{v0}+\xi_{v2})+g_{1}(\xi_{u0}+\xi_{u2}))+2f_{02}g_{1}(\xi_{v0}+\xi_{v2})\\
&~~~~~~-(\eta/\kappa)f_{1}(\xi_{u0}+\xi_{u2}e^{-3k_{T}^{2}\sigma^{2}/2})-(\eta/\kappa)f_{1}(\xi_{u0}+\xi_{u2})e^{-k_{T}^{2}\sigma^{2}/2}+6F_{4},\\
G_{3} &= 2g_{20}f_{1}(\xi_{u0}+\xi_{u2})+g_{11}(f_{1}(\xi_{v0}+\xi_{v2})+g_{1}(\xi_{u0}+\xi_{u2}))+2g_{02}g_{1}(\xi_{v0}+\xi_{v2})+6G_{4},\\
F_{4}&=f_{1}^{3}f_{30}+f_{1}^{2}g_{1}f_{21}+f_{1}g_{1}^{2}f_{12}+g_{1}^{3}f_{03},\\
G_{4}&=f_{1}^{3}g_{30}+f_{1}^{2}g_{1}g_{21}+f_{1}g_{1}^{2}g_{12}+g_{1}^{3}g_{03}.
\end{align*}

From the equations (\ref{SONLM}), (\ref{PST}), (\ref{COFOS}), and (\ref{COSOS}), we find the relation between the amplitudes as
\begin{equation}{\label{AEUC}}
 \begin{pmatrix}
A_{j,u} \\
A_{j,v}
\end{pmatrix}= \epsilon \begin{pmatrix}
f_{1} \\
g_{1}
\end{pmatrix} W_{j} + \epsilon^{2} \begin{pmatrix}
f_{1} \\
g_{1}
\end{pmatrix} Z_{j} + o(\epsilon^{3}).
\end{equation}

\noindent Now, we focus on the amplitude equations corresponding to the $u$-component. For the notational simplicity, we denote $A_{j,u}$ as $A_{j}$. Therefore, from (\ref{AEUC}), we have
\begin{equation}
A_{j} = \epsilon f_{1}W_{j} + \epsilon^{2}f_{1}Z_{j} + o(\epsilon^{3}),~1 \leq j \leq 3.
\end{equation}
Then the amplitude equation with three nodes $A_{j} (j=1,2,3)$ is given by
\begin{subequations}{\label{GAE}}
\begin{align}
\tau_{0}\frac{\partial A_{1}}{\partial t} &= \mu A_{1} + h_{0}\overline{A}_{2}\overline{A}_{3}-(m_{1}|A_{1}|^{2}+m_{2}(|A_{2}|^{2}+|A_{3}|^{2}))A_{1},\\
\tau_{0}\frac{\partial A_{2}}{\partial t} &= \mu A_{2} + h_{0}\overline{A}_{1}\overline{A}_{3}-(m_{1}|A_{2}|^{2}+m_{2}(|A_{1}|^{2}+|A_{3}|^{2}))A_{2}, \\
\tau_{0}\frac{\partial A_{3}}{\partial t} &= \mu A_{3} + h_{0}\overline{A}_{1}\overline{A}_{2}-(m_{1}|A_{3}|^{2}+m_{2}(|A_{1}|^{2}+|A_{2}|^{2}))A_{3},
\end{align}
\end{subequations}
where $\mu =(d_{T}-d)/d_{T}$ is a normalized distance to onset of the Turing bifurcation threshold, $\tau_{0} = (f_{1}f_{2}+g_{1}g_{2})/d_{T}k_{T}^{2}g_{1}g_{2}$ is a typical relaxation time, $h_{0}=2(f_{2}F_{1}+g_{2}G_{1})/d_{T}k_{T}^{2}f_{1}g_{1}g_{2}$, $m_{1}=-(f_{2}F_{2}+g_{2}G_{2})/d_{T}k_{T}^{2}f_{1}^{2}g_{1}g_{2}$ and $m_{2}=-(f_{2}F_{3}+g_{2}G_{3})/d_{T}k_{T}^{2}f_{1}^{2}g_{1}g_{2}$. Note that the amplitude equations corresponding to the $v$-component are also like the equations for the $u$-component. For the case of $v$-component, the expression of the coefficients $\mu$ and $\tau_{0}$ remain unaltered and the other coefficients are give by $h_{0}=2(f_{2}F_{1}+g_{2}G_{1})/d_{T}k_{T}^{2}g_{1}^2g_{2}$, $m_{1}=-(f_{2}F_{2}+g_{2}G_{2})/d_{T}k_{T}^{2}g_{1}^{3}g_{2}$ and $m_{2}=-(f_{2}F_{3}+g_{2}G_{3})/d_{T}k_{T}^{2}g_{1}^{3}g_{2}$.

Now, each of the amplitudes in equation (\ref{GAE}) can be decomposed into the mode $\rho_{j} = |A_{j}| (j=1,2,3)$ and a corresponding phase $\phi_{j}$. Substituting $A_{j}=\rho_{j}\mbox{exp}(i\phi_{j})$ into equations of (\ref{GAE}) and separating the real and imaginary parts, we obtain the following differential equations in real variables:
\begin{subequations}{\label{GAS}}
\begin{align}
\tau_{0}\frac{\partial \Phi}{\partial t} &= -h_{0}\frac{\rho_{1}^{2}\rho_{2}^{2}+\rho_{2}^{2}\rho_{3}^{2}+\rho_{3}^{2}\rho_{1}^{2}}{\rho_{1}\rho_{2}\rho_{3}} \sin \Phi,\label{GASa}\\
\tau_{0}\frac{\partial \rho_{1}}{\partial t} &= \mu \rho_{1} + h_{0}\rho_{2}\rho_{3}\cos \Phi-m_{1}\rho_{1}^{3} -m_{2}(\rho_{2}^{2}+\rho_{3}^{2})\rho_{1}, \label{GASb}\\
\tau_{0}\frac{\partial \rho_{2}}{\partial t} &= \mu \rho_{2} + h_{0}\rho_{1}\rho_{3}\cos \Phi-m_{1}\rho_{2}^{3} -m_{2}(\rho_{1}^{2}+\rho_{3}^{2})\rho_{2}, \label{GASc}\\
\tau_{0}\frac{\partial \rho_{3}}{\partial t} &= \mu \rho_{3} + h_{0}\rho_{1}\rho_{2}\cos \Phi-m_{1}\rho_{3}^{3} -m_{2}(\rho_{1}^{2}+\rho_{2}^{2})\rho_{3},\label{GASd}
\end{align}
\end{subequations}
where $\Phi = \phi_{1}+\phi_{2}+\phi_{3}$. The system (\ref{GAS}) has always an equilibrium point with the components $\rho_{1} = \rho_{2} = \rho_{3} = 0$. If this trivial equilibrium point is stable, then following (\ref{SONLM}), we can not find any non-homogeneous stationary pattern (Turing pattern) for the nonlocal model (\ref{NLM}). Therefore, for the generation of a non-homogeneous stationary pattern, the trivial equilibrium point of the system (\ref{GAS}) has to be unstable.

Now, we find all the equilibrium points of the system (\ref{GAS}) and determine their stabilities. For non-zero amplitudes $\rho_{1}$, $\rho_{2}$, $\rho_{3}$ and $\tau_{0} > 0$, the solution corresponding to $\Phi = 0$ ($H_{0}$ pattern) is stable if $h_{0}>0$, and the solution corresponding to $\Phi = \pi$ ($H_{\pi}$ pattern) is stable if $h_{0}<0$. The mode equations for the stable solutions of the equation (\ref{GASa}) is given by
\begin{subequations}{\label{GASM}}
\begin{align}
\tau_{0}\frac{\partial \rho_{1}}{\partial t} &= \mu \rho_{1} + |h_{0}|\rho_{2}\rho_{3}-m_{1}\rho_{1}^{3} -m_{2}(\rho_{2}^{2}+\rho_{3}^{2})\rho_{1}, \label{GASMA}\\
\tau_{0}\frac{\partial \rho_{2}}{\partial t} &= \mu \rho_{2} + |h_{0}|\rho_{1}\rho_{3}-m_{1}\rho_{2}^{3} -m_{2}(\rho_{1}^{2}+\rho_{3}^{2})\rho_{2}, \label{GASMB}\\
\tau_{0}\frac{\partial \rho_{3}}{\partial t} &= \mu \rho_{3} + |h_{0}|\rho_{1}\rho_{2}-m_{1}\rho_{3}^{3} -m_{2}(\rho_{1}^{2}+\rho_{2}^{2})\rho_{3}.\label{GASMC}
\end{align}
\end{subequations}

\begin{theorem}\label{thm6}
Suppose $\tau_{0}>0$. Then system (\ref{NLM}) admits four kinds of solutions, and they are
\begin{itemize}
\item[(I)] Homogeneous solution: $\rho_{1} = \rho_{2} = \rho_{3} = 0$. It is stable for $\mu < \mu_{2} = 0$ and unstable for $\mu > \mu_{2}$.
\item[(II)] Stripe pattern:
$\rho_{1} = \sqrt{\mu/m_{1}} \neq 0,\rho_{2} = \rho_{3} = 0$. It is stable for $\mu > \mu_{3} = h_{0}^{2}m_{1}/(m_{2}-m_{1})^{2}$ and unstable for $\mu < \mu_{3}$.
\item[(III)] Hexagonal pattern ($H_{0}$ or $H_{\pi}$):
\begin{equation*}
\rho_{1} = \rho_{2} = \rho_{3} = \frac{|h_{0}|\pm \sqrt{h_{0}^{2}+4(m_{1}+2m_{2})\mu}}{2(m_{1}+2m_{2})} \equiv \rho^{\pm}.
\end{equation*}
These amplitudes exist if $\mu > \mu_{1} = -h_{0}^{2}/(4(m_{1}+2m_{2}))$. The solution $\rho^{+} $ is stable for $\mu < \mu_{4}=h_{0}^{2}(2m_{1}+m_{2})/(m_{2}-m_{1})^{2}$, and $\rho^{-}$ is always unstable.
\item[(IV)] Mixed pattern:
\begin{equation*}
\rho_{1} =\frac{|h_{0}|}{m_{2}-m_{1}},~~ \rho_{2} = \rho_{3} = \sqrt{\frac{\mu -m_{1}\rho_{1}^{2}}{m_{1}+m_{2}}},
\end{equation*}
with $m_{2}>m_{1}$, $\mu > m_{1}\rho_{1}^{2}$, and they are always unstable.
\end{itemize}
\end{theorem}
\begin{proof}
Let $(\delta \rho_{1}, \delta \rho_{2}, \delta \rho_{3})$ be a perturbation of $(\rho_{1}, \rho_{2}, \rho_{3})$. Putting these perturbations in the system (\ref{GASM}) and ignoring the second and higher order terms, we obtain the following matrix equation
\begin{equation}
    \tau_{0}\frac{d}{dt} \begin{bmatrix}
     \delta \rho_{1} \\
     \delta \rho_{2} \\
     \delta \rho_{3}
    \end{bmatrix} = \begin{bmatrix}
     \omega_{11} & \omega_{12} & \omega_{13} \\
     \omega_{21} & \omega_{22} & \omega_{23} \\
     \omega_{31} & \omega_{32} & \omega_{33} 
    \end{bmatrix} \begin{bmatrix}
     \delta \rho_{1} \\
     \delta \rho_{2} \\
     \delta \rho_{3}
    \end{bmatrix}\equiv \mathbf{W}\begin{bmatrix}
     \delta \rho_{1} \\
     \delta \rho_{2} \\
     \delta \rho_{3}
    \end{bmatrix},
\end{equation}
where $\omega_{11} = \mu -3m_{1}\rho_{1}^{2}-m_{2}(\rho_{2}^{2}+\rho_{3}^{2})$, $\omega_{12} = |h_{0}|\rho_{3}-2m_{2}\rho_{1}\rho_{2}$, $\omega_{13} = |h_{0}|\rho_{2}-2m_{2}\rho_{1}\rho_{3}$, $\omega_{21} = |h_{0}|\rho_{3}-2m_{2}\rho_{1}\rho_{2}$, $\omega_{22} = \mu -3m_{1}\rho_{2}^{2}-m_{2}(\rho_{1}^{2}+\rho_{3}^{2})$, $\omega_{23} = |h_{0}|\rho_{1}-2m_{2}\rho_{2}\rho_{3}$, $\omega_{31} = |h_{0}|\rho_{2}-2m_{2}\rho_{1}\rho_{3}$, $\omega_{32} = |h_{0}|\rho_{1}-2m_{2}\rho_{2}\rho_{3}$ and $\omega_{33} = \mu -3m_{1}\rho_{3}^{2}-m_{2}(\rho_{1}^{2}+\rho_{2}^{2})$.

\noindent \textbf{Case (I):} If $\rho_{1} = \rho_{2} = \rho_{3} = 0$, then $\mathbf{W}$ becomes
\begin{equation*}
    \mathbf{W} = \begin{bmatrix}
     \mu & 0 & 0\\
     0 & \mu & 0\\
     0 & 0 & \mu
    \end{bmatrix}
\end{equation*}
Therefore, the homogeneous solution is stable if $\mu < \mu_{2} = 0$ and unstable for $\mu > \mu_{2}$.\\
\textbf{Case (II):} If $\rho_{1} = \sqrt{\mu/m_{1}} \neq 0,\rho_{2} = \rho_{3} = 0$, then  the eigenvalues of the matrix $\mathbf{W}$ are $$-2\mu, \mu \bigg{(}1-\frac{m_{2}}{m_{1}}\bigg{)}\pm |h_{0}|\sqrt{\frac{\mu}{m_{1}}}.$$
Therefore, the stripe pattern is stable if $m_{2}>m_{1}>0$ and $\mu > \mu_{3} = h_{0}^{2}m_{1}/(m_{2}-m_{1})^{2}$. Similarly, we can prove the rest of the cases.
\end{proof}

\section{Numerical Results}{\label{SE3}}

This section validates all the obtained theoretical results through numerical computations. First, we present some bifurcations for the temporal model, then the spatio-temporal pattern for the local model, followed by the patterns for the nonlocal model. For the numerical simulations of the local and nonlocal models, we have used the SHARCNET (www.sharcnet.ca) high performance computational facilities to minimize the time in computations. 

\begin{figure}[H]
        \centering
        \includegraphics[width=10cm]{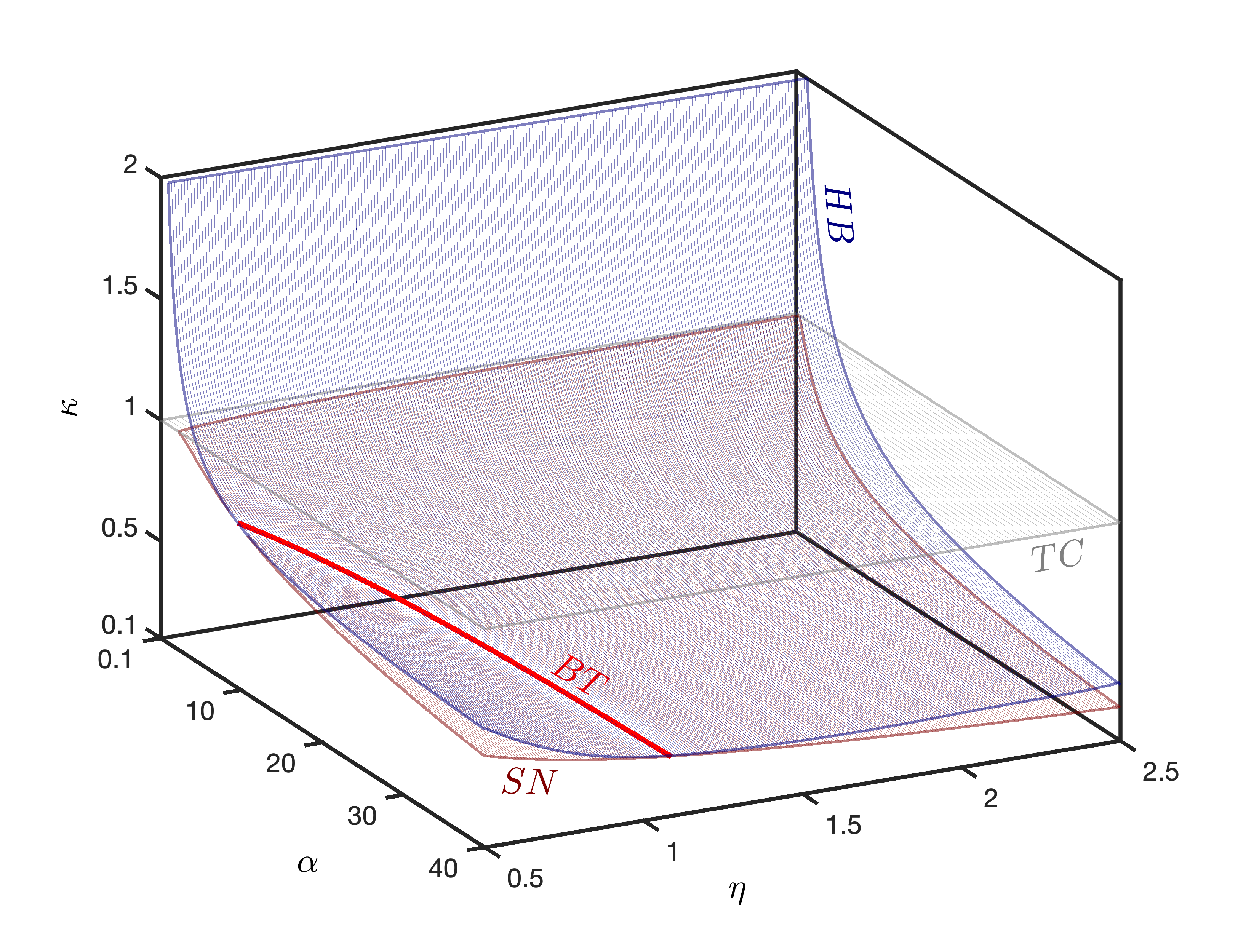}
\caption{(Color online) Three parametric bifurcation diagram for the temporal model (\ref{model1}). Bifurcation surfaces are transcritical (TC), saddle-node (SN), and Hopf (HB). The intersection of two surfaces SN and HB is the Bogdanov-Takens (BT) bifurcation.}\label{fig:B3d}
\end{figure}

\subsection{Temporal bifurcations}

Figure \ref{fig:B3d} represents a bifurcation diagram for the temporal model in three-dimensional parameter space. We find three local bifurcation surfaces, namely saddle-node (SN), transcritical (TC) and Hopf (HB) bifurcation surfaces. Among them, the transcritical bifurcation surface is a plane $\kappa = 1$. The saddle-node bifurcation surface lies below the transcritical bifurcation plane. The Hopf bifurcation surface lies both above and below the plane $\kappa=1$. The Hopf bifurcation and saddle-node bifurcation surfaces touch each other along the Bogdanov-Takens (BT) bifurcation curve (marked in red). There is no coexistence equilibrium point for parameter values lying below the SN surface. We find two coexistence equilibrium points for parameter values lying between SN and TC surfaces. One of the coexistence equilibrium points disappears through a transcritical bifurcation as $\kappa$ crosses the value 1 from below.  

SN, TC, and HB are the global bifurcations for the temporal model; however, our primary goal is to match the analytical pattern selections with the numerical non-homogeneous stationary patterns in the Turing domain. Here, HB plays an essential role as for the existence of Turing bifurcation, the non-trivial equilibrium point of the temporal model has to be Hopf-stable. So, we choose the parameter values so that the non-trivial equilibrium point of the temporal model satisfies the necessary condition for the Turing bifurcation. We discuss all the parameter values in the next section.

\subsection{Patterns for the local model}

In this subsection, we report extensive numerical simulations for the local and nonlocal models and present some qualitative results. We use periodic boundary conditions for both the models (local and nonlocal) in a square domain $[-25, 25]\times [-25,25]$ with space steps $dx = dy = 0.25$ and a time step $dt=0.01$. We have verified the numerical artifact of the patterns by choosing smaller space and time discretization steps for the simulations, and the presented patterns are unaltered. For finding the numerical solution of the local model, we have used Euler's method for the time derivative and a second-order finite-difference scheme to approximate the diffusion terms. Heterogeneous perturbations around the homogeneous steady-state have been used for the initial conditions for the species $u$ and $v$. We run the simulations in each of the upcoming patterns until the solution characteristics stop changing over time. We choose $\kappa =0.65$ and $\alpha = 10$ to be fixed parameters (unless stated otherwise) and $\eta$ as a bifurcation parameter. 

We first analyze the pattern formation for the local model. Taking the fixed parameter values, as mentioned above, we choose the bifurcation parameter value $\eta=0.92$. For this specific choice of $\eta$, the temporal system has two coexisting equilibrium points: one is a saddle point and the other is locally asymptotically stable. As mentioned earlier, we choose the locally asymptotically stable equilibrium point for finding the Turing bifurcation. For the local model, i.e., $\sigma =0$, we solve the equation (\ref{TBCK}) for $k$, and we find the critical wave number as $k_{T}=0.871$ and the corresponding Turing bifurcation threshold as $d_{T}=0.2715$. For the local model, the region $d<d_{T}$ is the Turing domain, where the stationary Turing pattern exists, and in the other region $d>d_{T}$, the homogeneous solution remains stable under the heterogeneous perturbation. 

\begin{figure}[H]
        \centering
        \includegraphics[width=10cm]{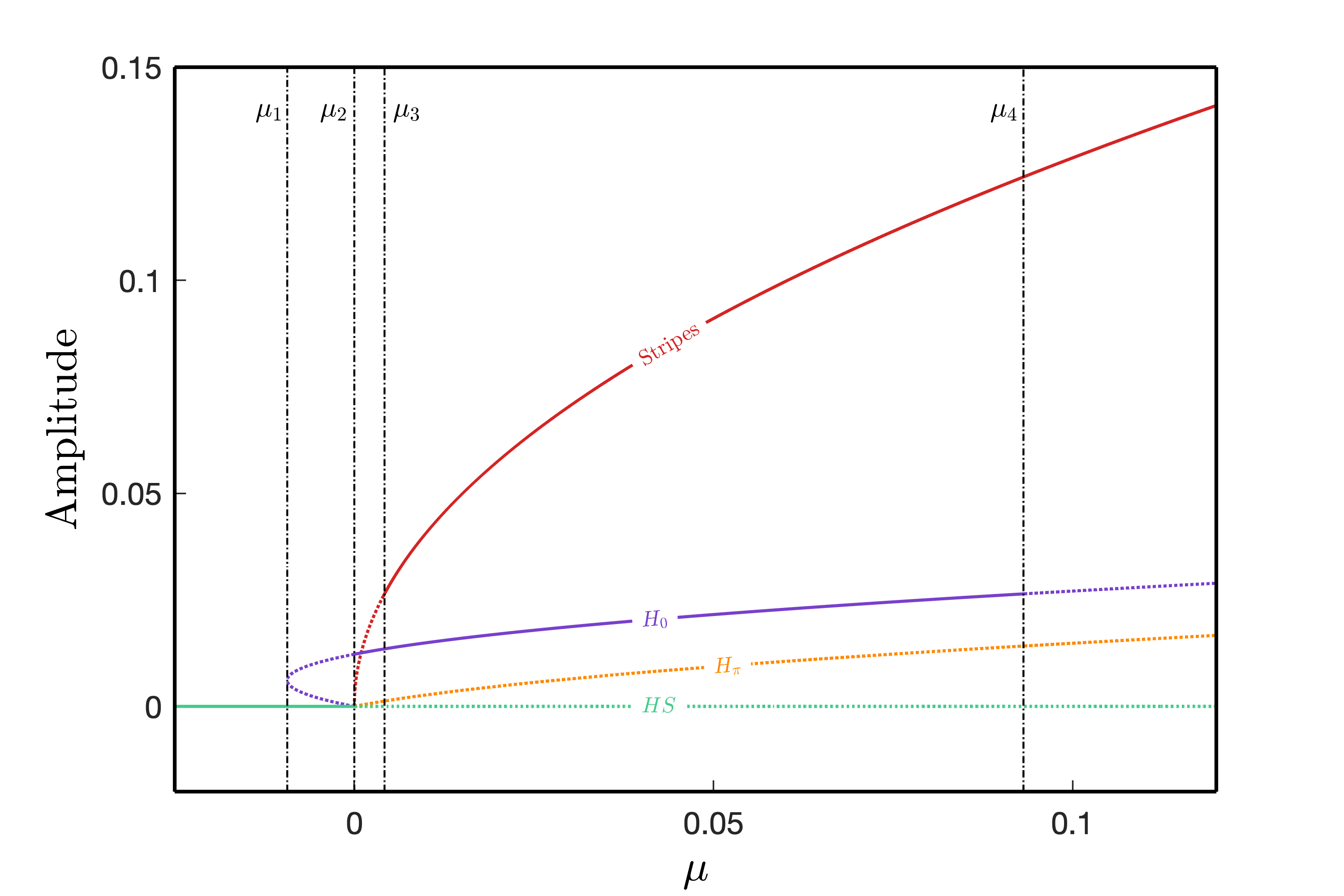}
\caption{(Color online) Bifurcation diagram of the patterns obtained through weakly nonlinear analysis for the local model with the parameter values $\kappa =0.65$, $\alpha = 10$ and $\eta = 0.92$. $HS$: homogeneous solution; $H_{0}$: hexagonal patterns with $\Phi=0$; $H_{\pi}$: hexagonal patterns with $\Phi=\pi$; solid curve: stable state; dashed line: unstable state.}\label{fig:mu_local}
\end{figure}

Now, for $\eta=0.92$, we find all the parameter values associated with the weakly nonlinear analysis and the values rounded up to three decimal places are $m_{1}=6.035$, $m_{2}=121.566$, $h_{0}=3.050$, $\tau_{0}=3.536$, $\mu_{1}=-0.009$, $\mu_{2}=0$ $\mu_{3}=0.004$, $\mu_{4}=0.093$. Figure \ref{fig:mu_local} depicts the stability of the stationary patterns for the local model obtained through the weakly nonlinear analysis. We find the diffusion coefficients $d_{j}$ ($j=1,2,3$ and $4$) corresponding to each $\mu_{j}$ by using the relation $d_{j} = (1-\mu_{j})d_{T}$ and the values are $d_{1} = 0.274$, $d_{2} = d_{T}$, $d_{3}=0.2704$, and $d_{4}=0.2462$. Next, we discuss all the results based on $d_{j}$. From Fig. \ref{fig:mu_local}, we can see that the homogeneous solution ($HS$) is stable for $d > d_{2}$. On the other hand, for $d_{3} < d < d_{2}$, two types of hexagonal patterns ($H_{0}$ and $H_{\pi}$) may occur in the system. Since, $\tau_{0}$ and $h_{0}$ both are positive, the hexagonal pattern $H_{0}$ (hot-spot) is stable, and the other one $H_{\pi}$ (cold-spot) is unstable (by Theorem \ref{thm6}). 


We find a hexagonal Turing pattern ($H_{0}$) for $d=0.271$ near the Turing bifurcation threshold in the Turing domain [see Fig. \ref{fig:pl1}(\subref{fig:pl1a})]. Here, the diffusion parameter $d$ lies in the bistable region $(d_{3},d_{2})$, hence the hexagonal pattern ($H_{0}$) is stable, and the other patterns are unstable, as predicted through weakly nonlinear analysis. Figure \ref{fig:pl1}(\subref{fig:pl1a}) depicts the stationary Turing pattern for the prey population. In this case, we observe a cold-spot pattern ($H_{\pi}$) corresponding to the predator populations (excluded in this paper). The relation of the patterns depends on the sign of $f_{1}$ and $g_{1}$, obtained through the weakly nonlinear analysis. If both of them are of the same sign, then the sign of $h_{0}$ remains the same for both components, and hence by Theorem \ref{thm6} both populations follow the same hexagonal patterns. Otherwise, the populations follow the opposite patterns because of the opposite signs in $h_{0}$. In our considered parameter values, $f_{1}$ is positive, and $g_{1}$ is negative [see Table \ref{table:WNAPV}], and hence the weakly nonlinear analysis is consistent with the pattern obtained through numerical simulation. The same type of negative correlation has been observed for all the upcoming stationary patterns. Moreover, we can find the amplitude equation corresponding to the predator population and it also predicts the same trend as discussed here.

\begin{figure}[H]
\begin{center}
        \begin{subfigure}[p]{0.32\textwidth}
                \centering
                \includegraphics[width=\textwidth]{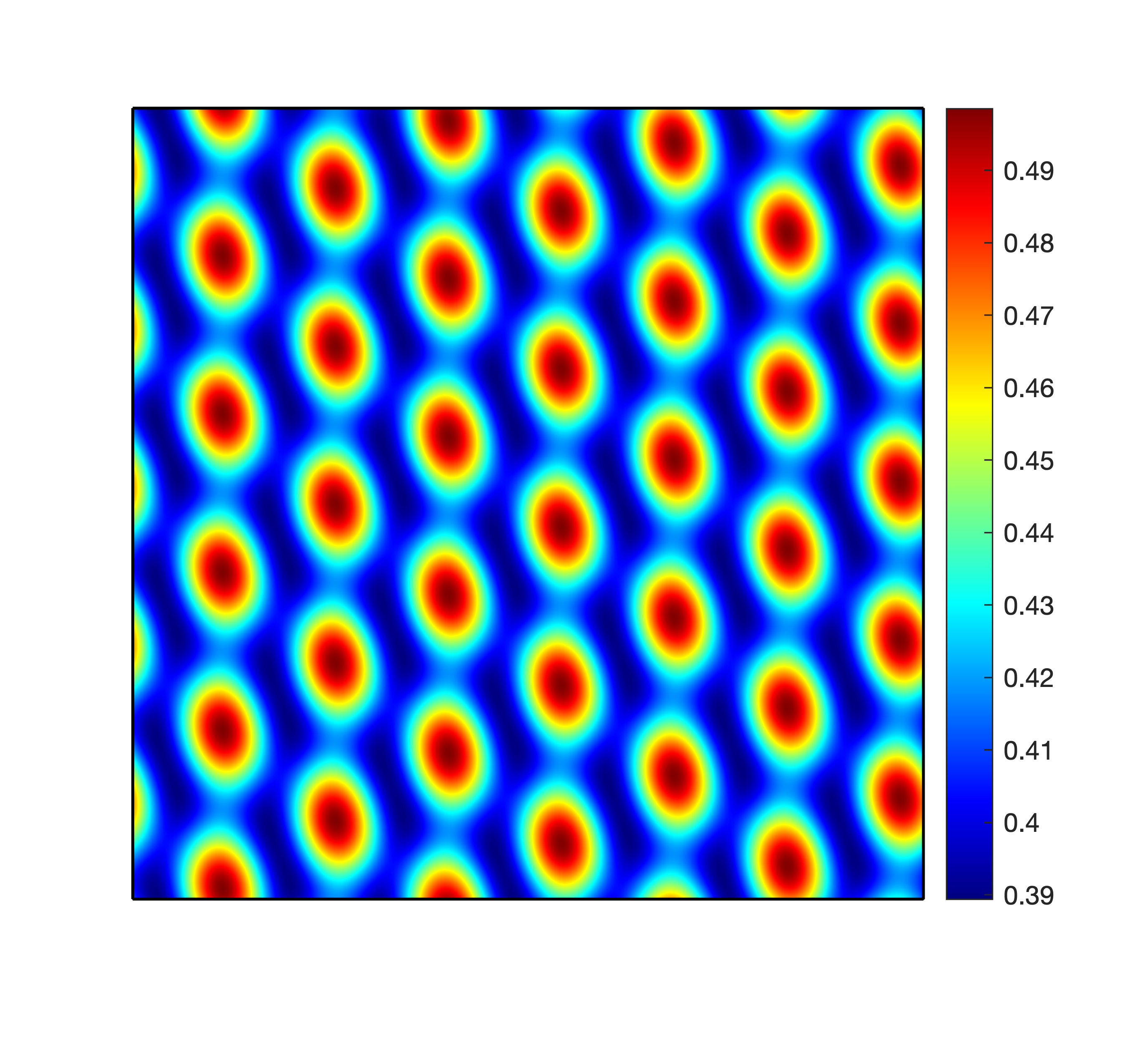}
                \caption{ }\label{fig:pl1a}
        \end{subfigure}%
        \begin{subfigure}[p]{0.32\textwidth}
                \centering
                \includegraphics[width=\textwidth]{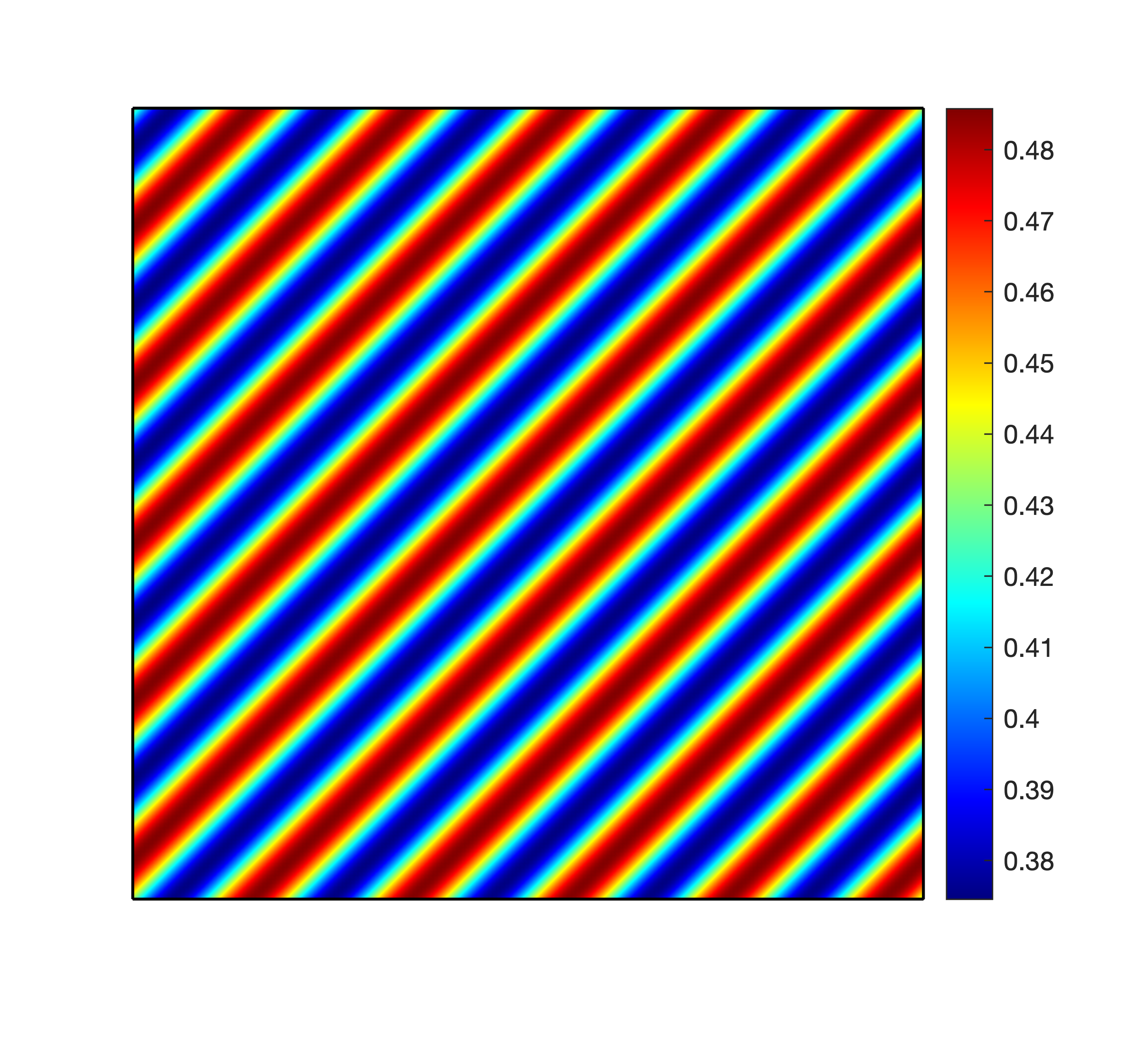}
                \caption{ }\label{fig:pl1b}
        \end{subfigure}%
        \begin{subfigure}[p]{0.32\textwidth}
                \centering
                \includegraphics[width=\textwidth]{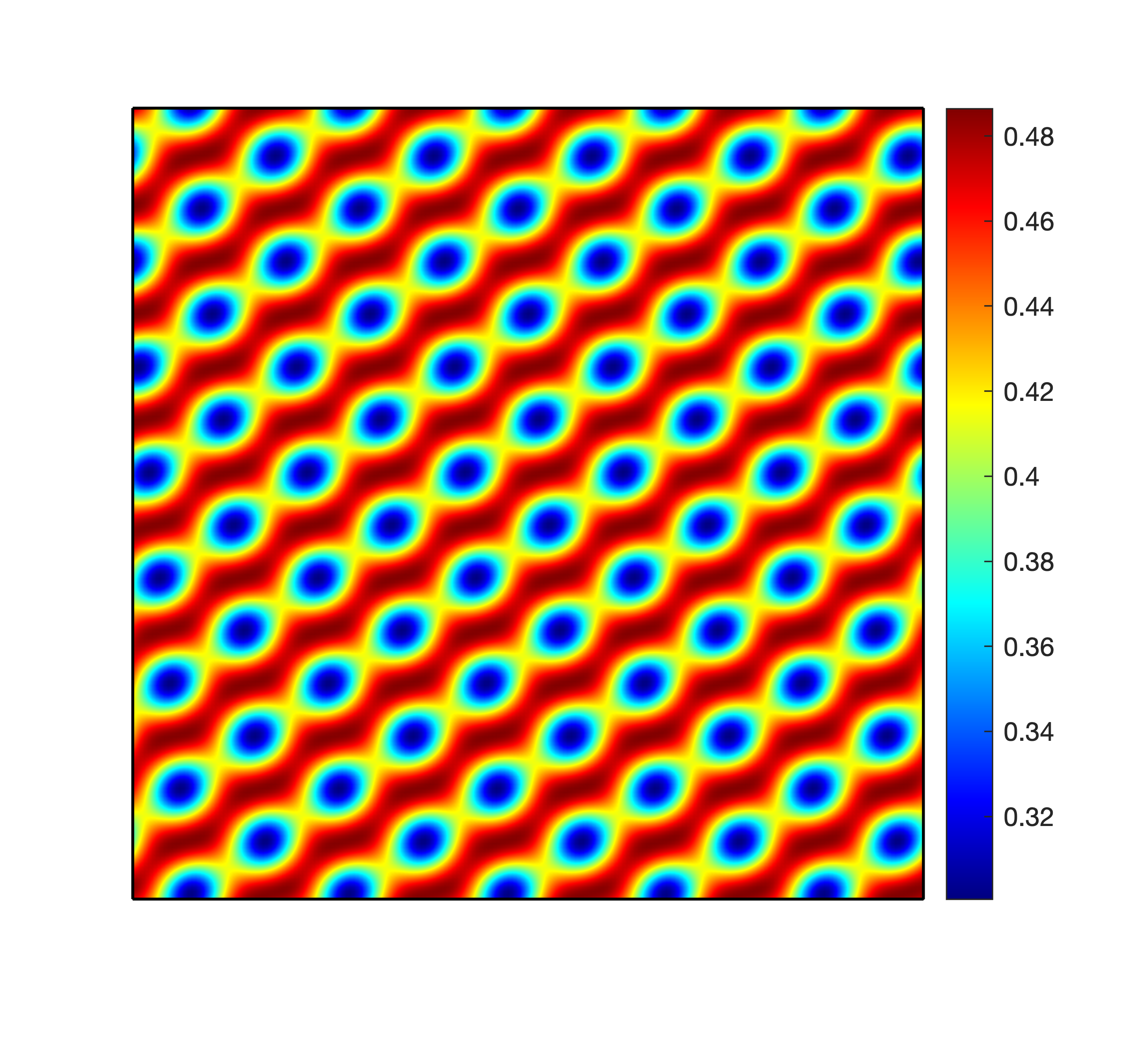}
                \caption{ }\label{fig:pl1c}
        \end{subfigure}%
\end{center}
\caption{(Color online) Turing patterns for the local model with $\eta = 0.92$, $\kappa =0.65$ and $\alpha = 10$: (a) $d =0.271$, (b) $d = 0.25$, and (c) $d=0.2$. In all the plots, $x$-axis and $y$-axis are  horizontal and vertical, respectively. }\label{fig:pl1}
\end{figure}

According to Theorem \ref{thm6}, the hexagonal pattern $H_{0}$ is stable for $d_{4}<d<d_{2}$ and the stripe pattern is stable for $d<d_{3}$ [see Fig. \ref{fig:mu_local}]. Therefore, the system has a bistable region $d_{4} < d < d_{3}$, where both patterns (hot-spot and stripe) are stable. We choose the diffusion parameter $d = 0.26$ in the bistable region $(d_{4},d_{3})$, and the corresponding stripe pattern is shown in Fig. \ref{fig:pl1}(\subref{fig:pl1b}). We will find the hot-spot pattern for different initial conditions.

Theorem \ref{thm6} predicts a stable stripe pattern for $d<d_{3}$. In the meantime, Fig. \ref{fig:pl1}(\subref{fig:pl1c}) depicts a cold-spot ($H_{\pi}$) pattern for $d=0.2$ in the prey population. This is called a reentry of a hexagonal pattern, and it occurs for the parameter values far from a Turing bifurcation threshold. In this case, some other primary slave modes become active, along with the critical wavenumber obtained through the linear stability analysis. We cannot neglect them in the derivation of the amplitude equations; rather, we should consider them in the amplitude equations \cite{ouyang2000}, and it would be an interesting extension of this work in the future. Combining all the results, with a decrease in the diffusion parameter value $d$, the number of stationary patches in a  spatial domain of fixed size increases.

\begin{table}[h!]
\caption{Computed parameter values of the local and nonlocal models for $\eta = 0.92$ corresponding to the prey population.}
\label{table:WNAPV}
\begin{center}
\begin{tabular}{|c|c|c|c|c|c|c|c|c|c|c|c|}
\hline
$\sigma$ & $d_{T}$ & $k_{T}$ & $f_{1}$  & $g_{1}$ & $m_{1}$ & $m_{2}$ & $h_{0}$ & $\tau_{0}$ &  $d_{3}$ & $d_{4}$ \\
\hline
0.00 & 0.2715 & 0.871 & 0.757 & -0.654 & 6.035 & 121.566 & 3.050 & 3.536 & 0.2704 & 0.2462 \\ 
0.25 & 0.2665 & 0.879 & 0.757 & -0.654 & 7.892 & 123.659 & 2.981 & 3.536 & 0.2651 & 0.2419 \\ 
0.50 & 0.2521 & 0.906 & 0.756 & -0.654 & 14.398 & 130.878 & 2.712 & 3.529 & 0.2502 & 0.2303 \\ 
0.75 & 0.2310 & 0.953 & 0.753 & -0.657 & 27.485 & 144.896 & 2.099 & 3.498 & 0.2289 & 0.2162 \\ 
1.00 & 0.2076 & 1.028 & 0.745 & -0.667 & 49.105 & 166.414 & 0.927 & 3.410 & 0.2070 & 0.2042 \\ 
1.25 & 0.1881 & 1.122 & 0.728 & -0.686 & 74.819 & 189.901 & -0.670 & 3.260 & 0.1876 & 0.1860 \\ 
1.50 & 0.1758 & 1.208 & 0.707 & -0.707 & 95.247 & 208.345 & -2.071 & 3.110 & 0.1702 & 0.1523 \\ 
\hline
\end{tabular}
\end{center}
\end{table}

\subsection{Patterns for the nonlocal model}

In this subsection, we focus on the numerical results of the analytical findings and the stationary Turing patterns for the nonlocal model. We have used the trapezoidal rule to find the numerical value of the convolution term. Figure \ref{fig:THNL} depicts the Turing bifurcation curves for the nonlocal model for different values of $\sigma$. With an increase in the parameter value of $\sigma$, the Turing bifurcation curve shifts downwards, however, it converges to a curve inside the considered domain in Fig. \ref{fig:THNL}. Theoretically, one can find the asymptotic curve of this Turing bifurcation curve (cf. \cite{pal2019}), but here we restrict the pattern formation for $\sigma \leq 1.5$. A detailed study of the pattern formation of the considered nonlocal model (\ref{NLM}) for higher values of $\sigma$ is one of the important issues for future work. 

\begin{figure}[H]
        \centering
        \includegraphics[width=8cm]{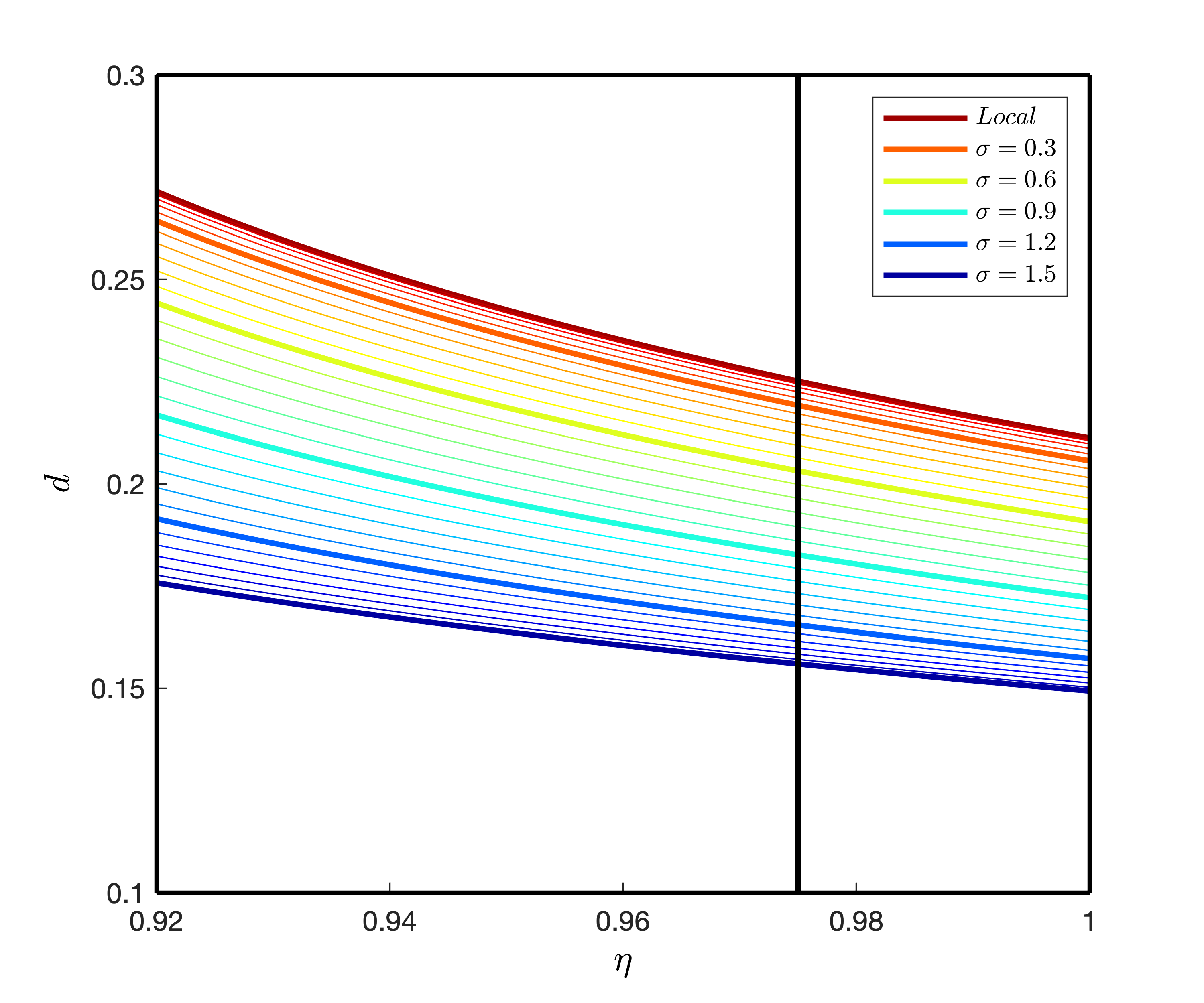}
\caption{(Color online) Turing and temporal-Hopf bifurcation curves for the local and nonlocal models.}\label{fig:THNL}
\end{figure}

\begin{figure}[H]
\begin{center}
        \begin{subfigure}[p]{0.32\textwidth}
                \centering
                \includegraphics[width=\textwidth]{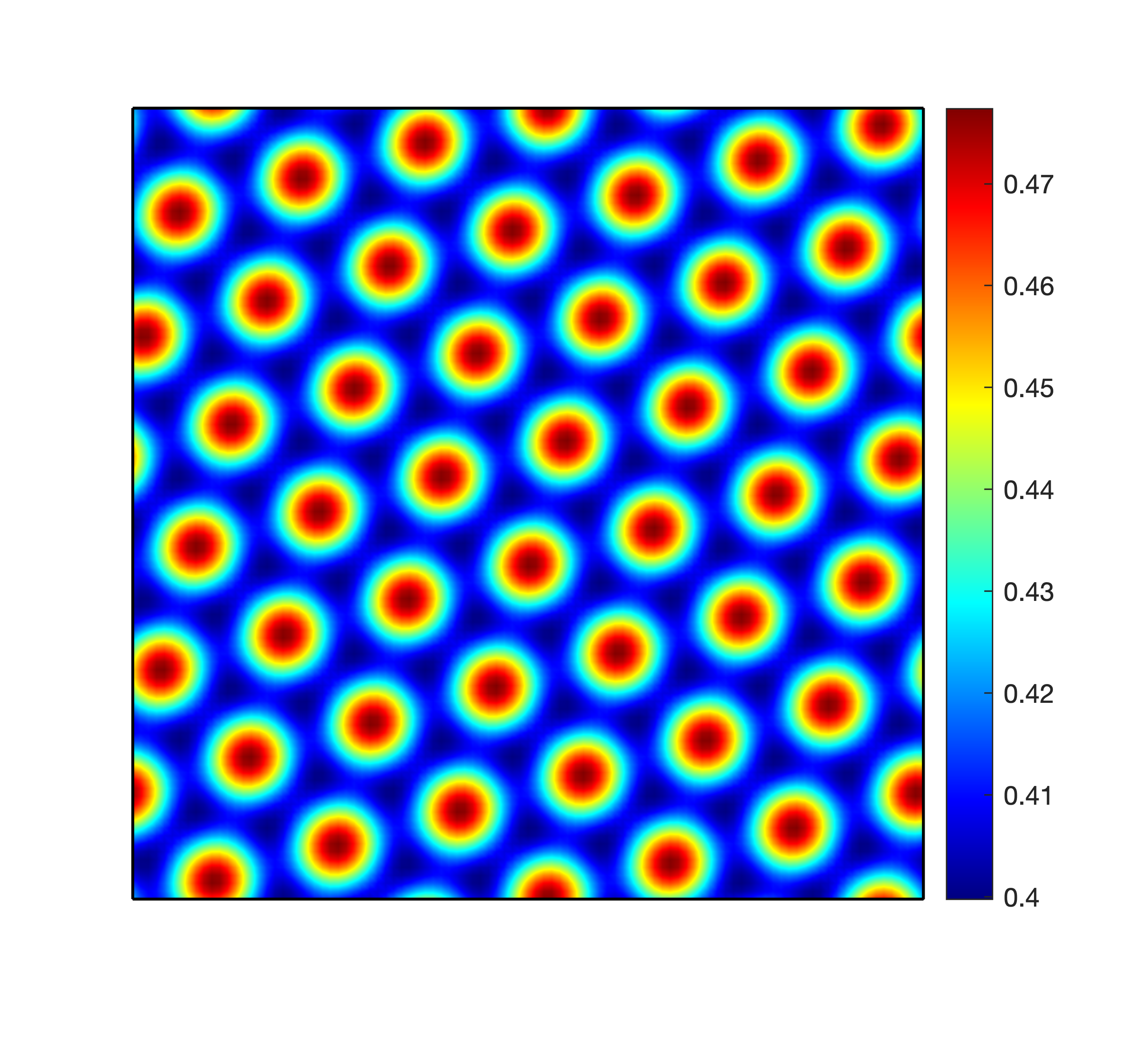}
                \caption{ }\label{fig:TP2a}
        \end{subfigure}%
        \begin{subfigure}[p]{0.32\textwidth}
                \centering
                \includegraphics[width=\textwidth]{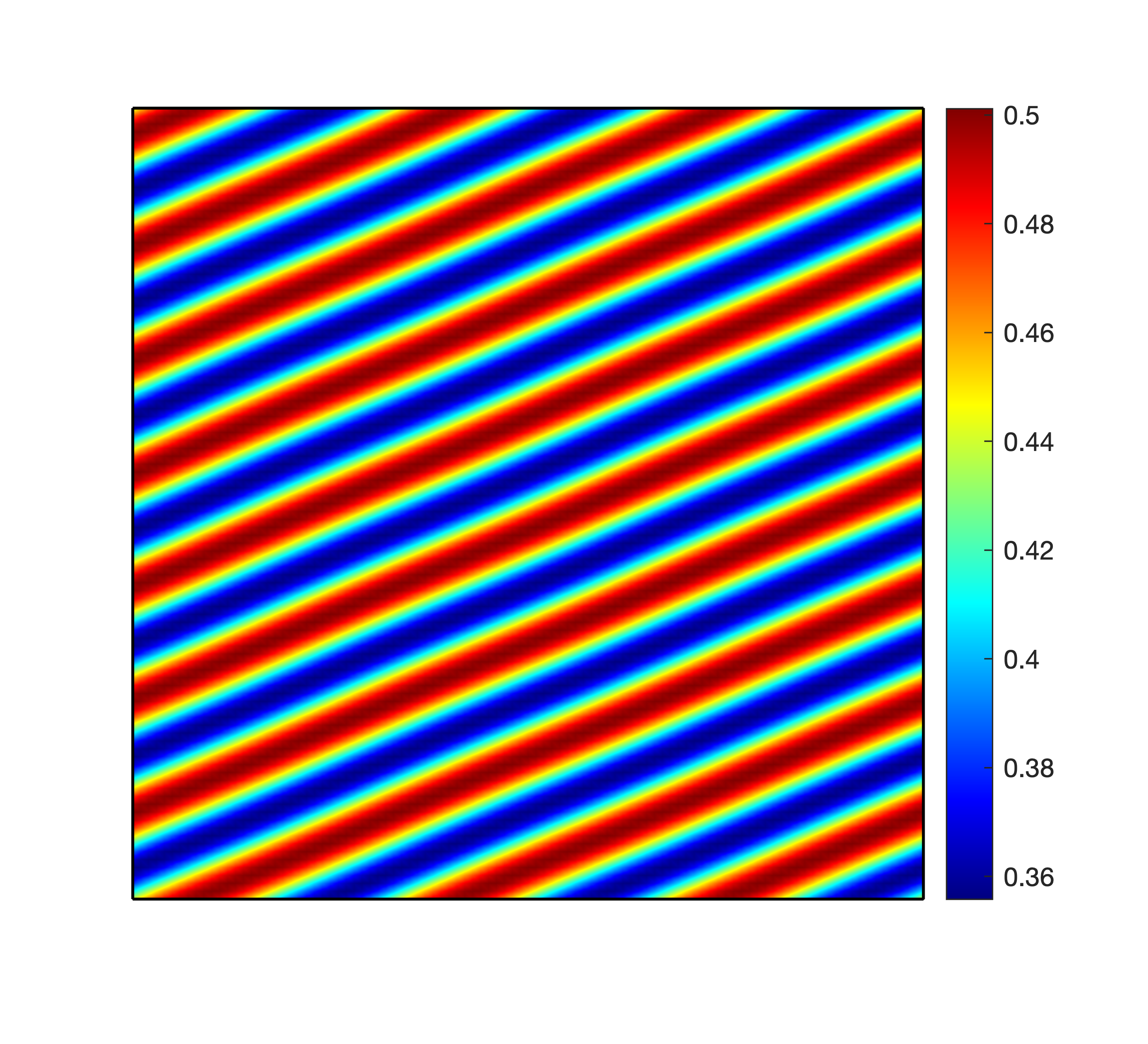}
                \caption{ }\label{fig:TP2b}
        \end{subfigure}%
        \begin{subfigure}[p]{0.32\textwidth}
                \centering
                \includegraphics[width=\textwidth]{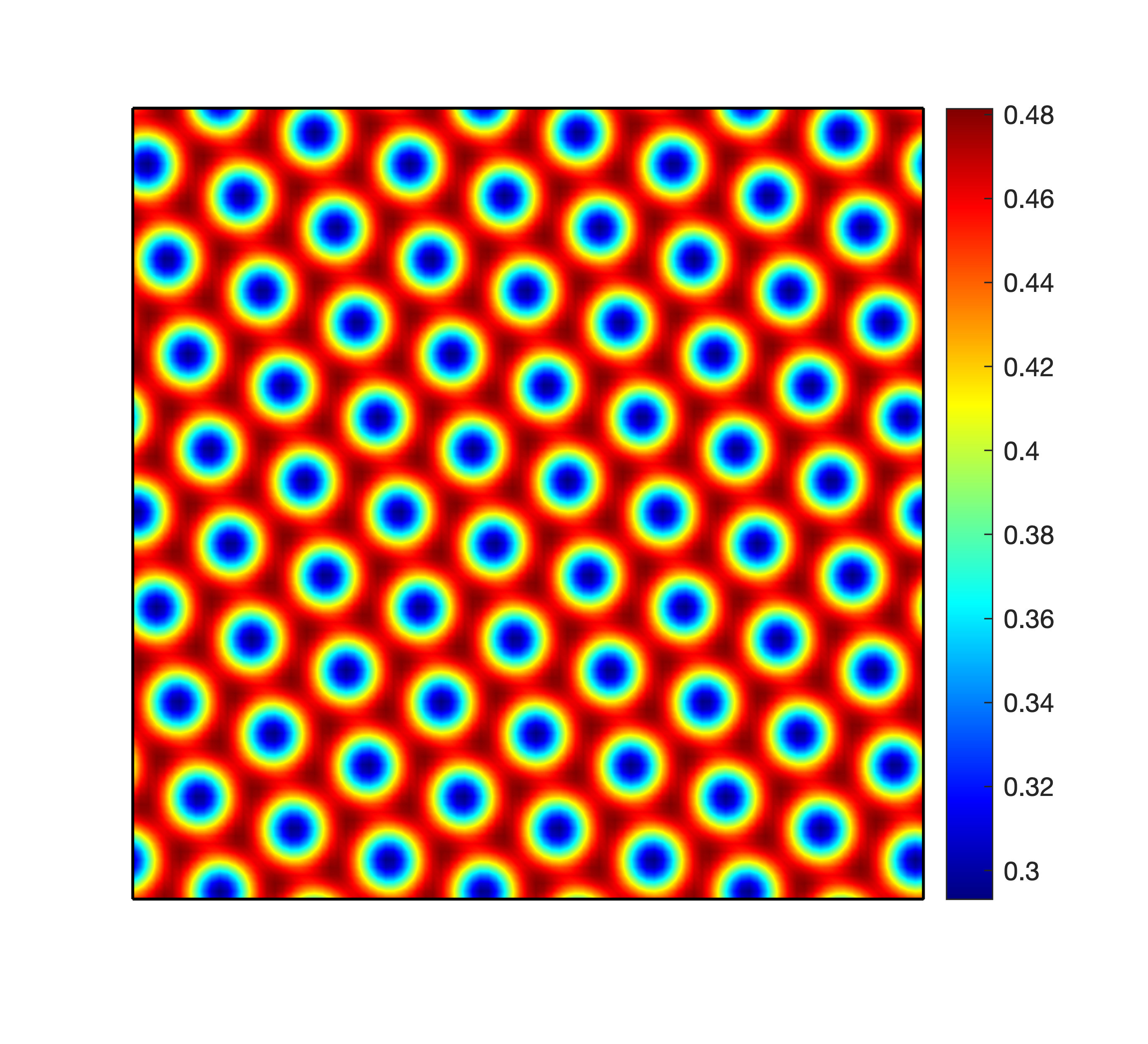}
                \caption{ }\label{fig:TP2c}
        \end{subfigure}%
\end{center}
\caption{(Color online) Turing patterns for the nonlocal model with $\eta = 0.92$, $\kappa =0.65$, $\alpha = 10$ and $\sigma = 0.5$: (a) $d =0.251$, (b) $d = 0.22$ and (c) $d=0.19$. In all the plots, $x$-axis and $y$-axis are horizontal and vertical, respectively.}\label{fig:TP2}
\end{figure}

We fix $\eta =0.92$ and $\sigma = 0.5$. For this choice of the parameter values, the Turing bifurcation threshold becomes $d_{T}=0.2521$ and the critical wavenumber $k_{T}=0.906$ [see Table \ref{table:WNAPV}]. Some of the computational parameter values for the weakly nonlinear analysis are listed in Table \ref{table:WNAPV}. Following Theorem \ref{thm6}, the cold spot pattern is stable for $d_{4}<d<d_{2}(=d_{T})$, the stripe pattern is stable for $d < d_{3}$. Numerical solutions of the nonlocal model for three different diffusion parameter values are plotted in Fig. 2, and they show a similar trend to the one obtained with the local model. However, the amplitudes of each of the patterns are different from the local model. Due to a similar reason discussed earlier, we can not predict the reentry of the hexagonal pattern (hot-spot) through the weakly nonlinear analysis described in this article.


\begin{figure}[H]
\begin{center}
        \begin{subfigure}[p]{0.32\textwidth}
                \centering
                \includegraphics[width=\textwidth]{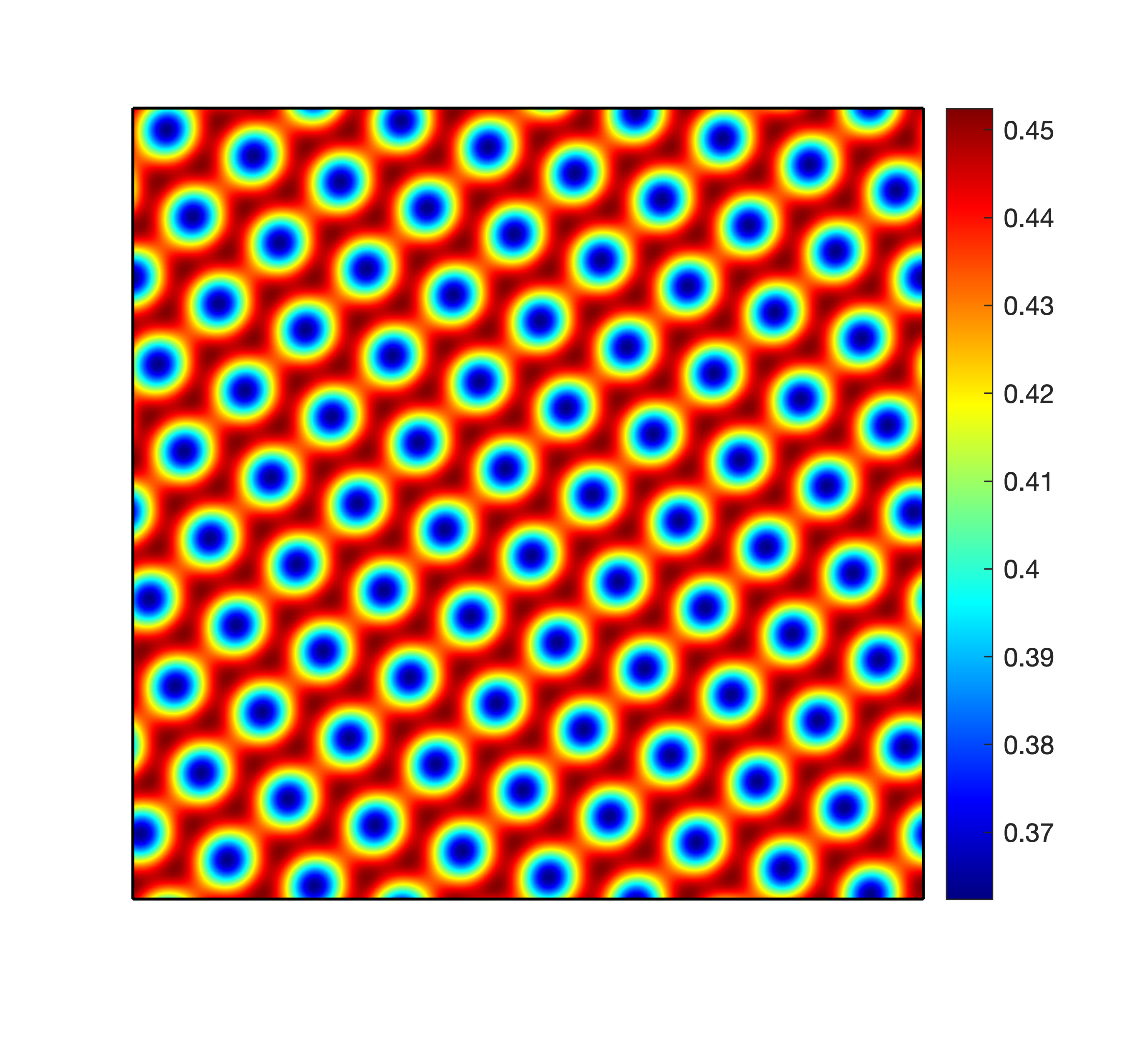}
                \caption{ }\label{fig:pl2a}
        \end{subfigure}%
        \begin{subfigure}[p]{0.32\textwidth}
                \centering
                \includegraphics[width=\textwidth]{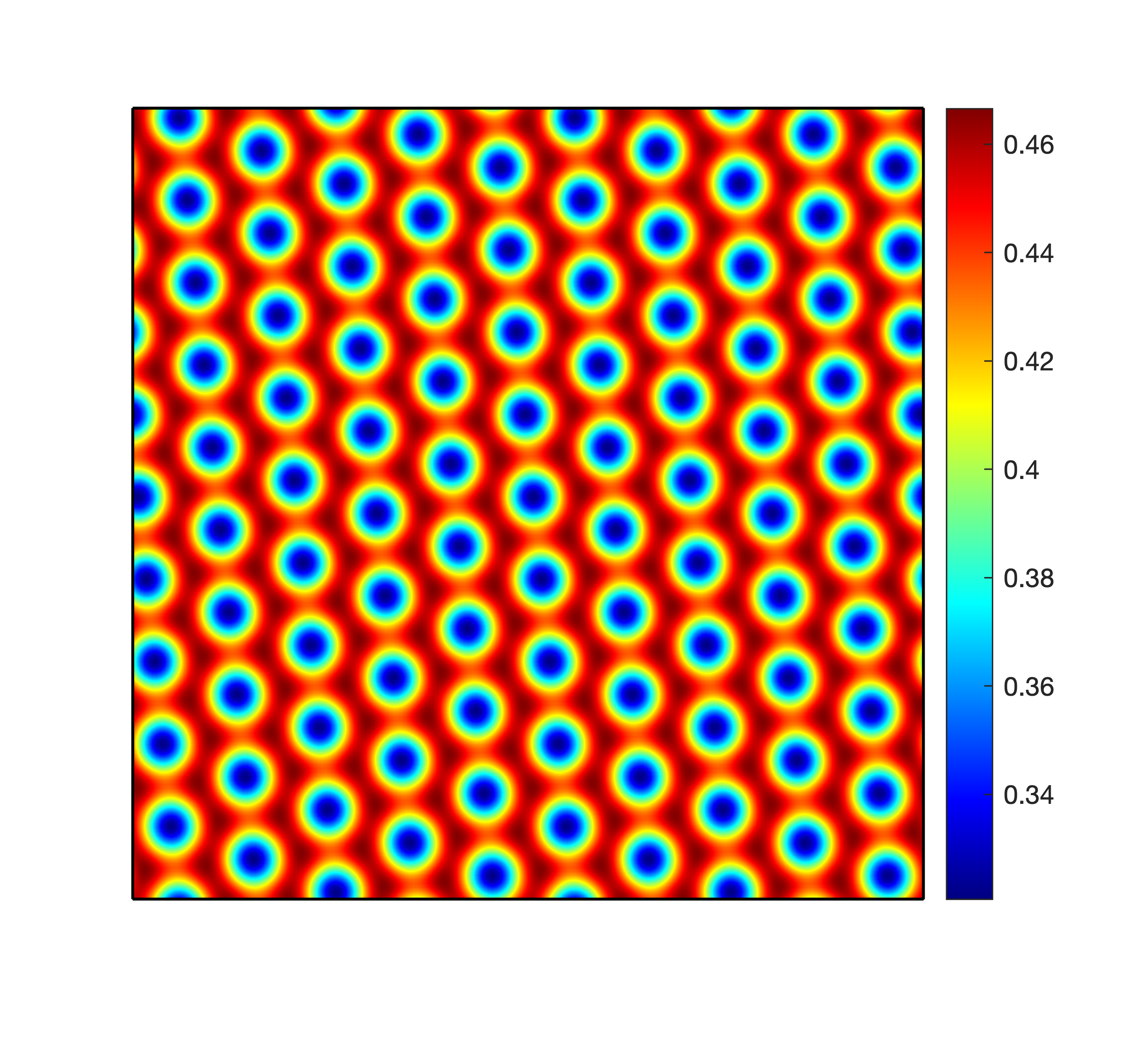}
                \caption{ }\label{fig:pl2b}
        \end{subfigure}%
        \begin{subfigure}[p]{0.32\textwidth}
                \centering
                \includegraphics[width=\textwidth]{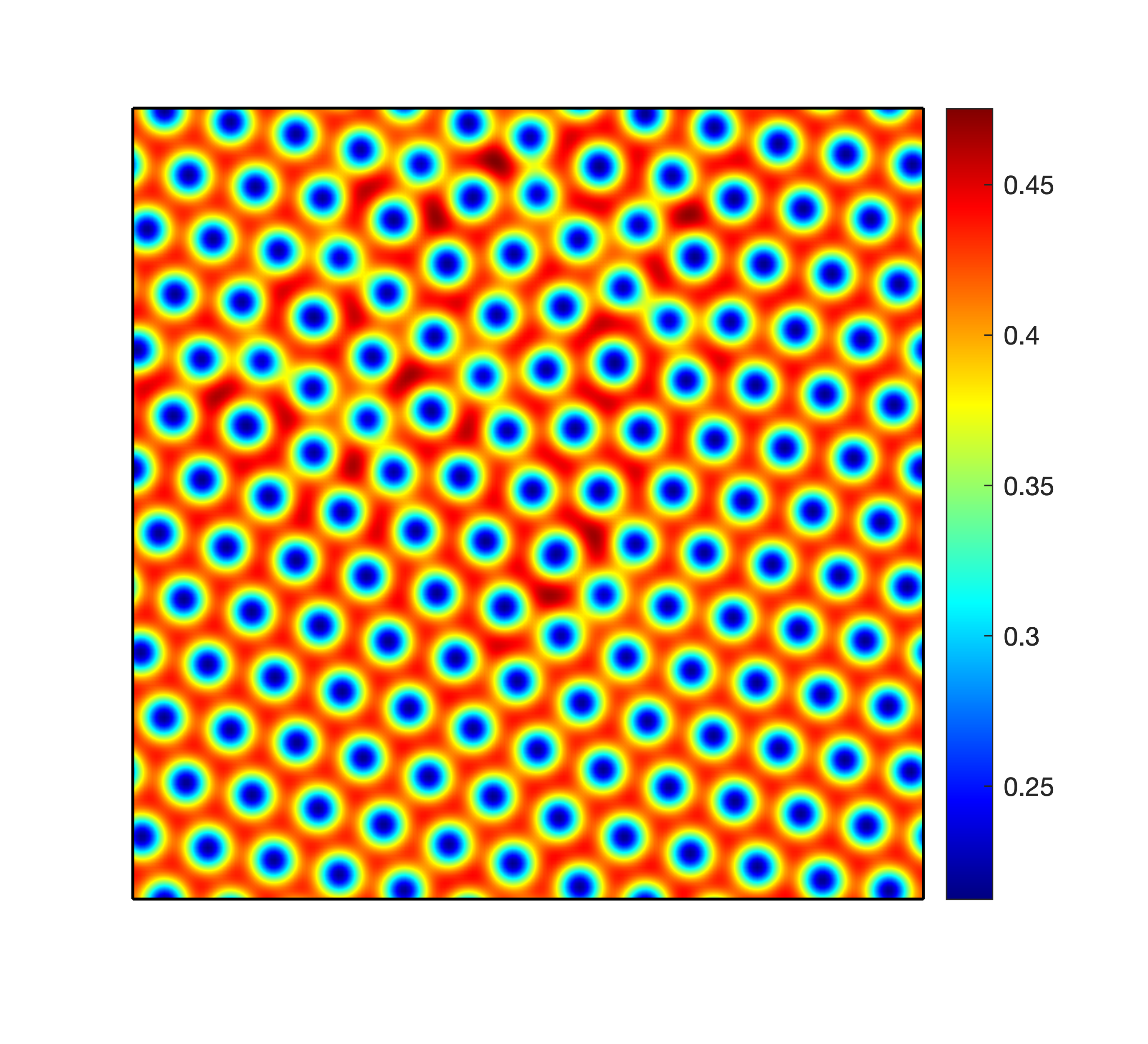}
                \caption{ }\label{fig:pl2c}
        \end{subfigure}%
\end{center}
\caption{(Color online) Turing patterns for the nonlocal model with $\eta = 0.92$, $\kappa =0.65$, $\alpha = 10$ and $\sigma = 1.5$: (a) $d =0.172$, (b) $d = 0.16$, and (c) $d=0.1$. In all the plots, $x$-axis and $y$-axis are horizontal and vertical, respectively. }\label{fig:pl2}
\end{figure}

Further, if we increase the parameter values of $\sigma$ in the mentioned range with fixing $\eta = 0.92$, the parameter $h_{0}$ changes its sign from positive to negative [see Table \ref{table:WNAPV}], but the other parameters remain with the same sign. This sign change in the parameter $h_{0}$ is reflected in the resulting pattern, e.g., we consider $\sigma = 1.5$, and the computed parameter values involved in the weakly nonlinear analysis is mentioned in Table \ref{table:WNAPV}. The numerical simulation result shows a cold-spot pattern for the prey population for the diffusion parameter $d = 0.172$ [see Fig. \ref{fig:pl2}(\subref{fig:pl2a})], near the Turing bifurcation threshold in the Turing domain. Figures \ref{fig:pl2}(\subref{fig:pl2b}) and (\subref{fig:pl2c}) show the cold-spot pattern for the lower diffusion parameters far from the Turing bifurcation threshold. Numerical simulation results show only the cold-spot pattern for the nonlocal model with $\sigma = 1.5$.

From Table \ref{table:WNAPV}, we see that the interval $(d_{4},d_{3})$ for the existence of the labyrinthine pattern becomes more narrow with an increase in the magnitude of $\sigma$. As a result, we observe that the thresholds corresponding to the hexagonal patterns ($d_{3}$ and the threshold for reentry of the hexagonal pattern) come closer, and then the labyrinthine pattern disappears. As a matter of fact, the amplitude equations fail to predict the selection of stripe pattern theoretically; however, it predicts the stationary hexagonal patterns near the Turing bifurcation threshold.

\section{Conclusions}

In this paper, we have studied the pattern formation for the local and nonlocal models, exemplifying our results on prey-predator models with hunting cooperation. Implicit analytical conditions have been derived for the existence of different temporal bifurcations for the diffusionless model, in particular, transcritical, saddle-node, and Hopf bifurcations. Moreover, a three-dimensional bifurcation diagram has been plotted to show these bifurcations numerically. The Turing bifurcation condition for the nonlocal model has been discussed analytically. Fixing the temporal parameters and the kernel parameter $\sigma$, we have derived the analytical equation for finding the critical wave number corresponding to the Turing bifurcation and this helps in finding the critical diffusion threshold. In particular, the Turing bifurcations for the local model are obtainable by taking the kernel parameter $\sigma = 0$.

The non-homogeneous stationary patterns exist for the local and nonlocal models in the Turing domain. The results of pattern selections can be identified through the weakly nonlinear analysis described in this manuscript. We have derived the amplitude equations for the nonlocal model near the Turing bifurcation threshold. Likewise, in the Turing bifurcation, the amplitude equations corresponding to the local model can be derived by taking $\sigma = 0$. We have compared the numerical simulation results with the theoretical results for the local model. They agreed near the Turing bifurcation threshold but differed far from the threshold. If we move from being close to the Turing bifurcation threshold to being further away the stationary pattern corresponding to the prey population changes from cold-spot $\rightarrow$ labyrinthine $\rightarrow$ hot-spot. The reentry of the hexagonal pattern (hot-spot) can not be predicted through the weakly nonlinear analysis. We have derived the amplitude equations for both prey and predator populations. With the help of weakly nonlinear analysis, we have shown that the prey and predator populations are negatively correlated with each other, and hence, the predator population shifts their patterns from hot-spot $\rightarrow$ labyrinthine $\rightarrow$ cold-spot when the parameter shifts being close to the Turing bifurcation threshold to being further away.

The pattern selection of the nonlocal model has also been discussed in this manuscript. With an increase in the nonlocal parameter $\sigma$, the Turing bifurcation curve shifts downwards but does not disappear; rather, it saturates to a curve. The weakly nonlinear analysis also predicts a tapered behaviour in the region where a labyrinthine pattern exists for the nonlocal model. Furthermore, numerical simulation results suggest the thresholds corresponding to reentry of the hexagonal pattern move towards the Turing bifurcation curve and eventually coincide with it. As a result, the stripe labyrinthine pattern disappears, and only the cold-spot hexagonal patterns exist for the nonlocal model. The reentry of the hexagonal pattern can not be identified through our described weakly nonlinear analysis due to neglecting some other primary slave modes, and it could be a possible extension of this work in the future. Finally, we note that the number of non-homogeneous stationary patches in a spatial domain of fixed size corresponding to a small value of the nonlocal parameter $\sigma$ is smaller compared to the higher values of $\sigma$.

\section*{Acknowledgements}
SP and RM are grateful to the NSERC and the CRC Program for their support. RM is also acknowledging support of the BERC 2022-2025 program and Spanish Ministry of Science, Innovation and Universities through the Agencia Estatal de Investigacion (AEI) BCAM Severo Ochoa excellence accreditation SEV-2017-0718 and the Basque Government fund AI in BCAM EXP. 2019/00432. This research was enabled in part by support provided by SHARCNET (www.sharcnet.ca) and Compute Canada (www.computecanada.ca).

\bibliography{References}

\end{document}